\documentclass{amsart}
\usepackage{amsxtra}
\usepackage{mathdots}
\usepackage{mathrsfs}
\usepackage{verbatim}
\usepackage{calc}
\usepackage{graphicx}
\usepackage{varwidth}
\usepackage{hyperref}
\usepackage[shortalphabetic,initials]{amsrefs}
\usepackage{bdsstandard}
\usepackage{pdfsync}

\theoremstyle{plain}
\newtheorem{thm}[equation]{Theorem}
\newtheorem*{thm*}{Theorem}
\newtheorem{prop}[equation]{Proposition}
\newtheorem*{prop*}{Proposition}

\newtheorem*{cor*}{Corollary}
\newtheorem{lem}[equation]{Lemma}
\newtheorem*{lem*}{Lemma}

\newtheorem*{conj*}{Conjecture}

\theoremstyle{definition}

\newtheorem*{defn*}{Definition}
\newtheorem{eg}[equation]{Example}
\newtheorem*{eg*}{Example}

\newtheorem*{ex*}{Exercise}
\newtheorem{rk}[equation]{Remark}

\renewcommand{\L}{\ensuremath{\mathscr{L}}\xspace}

\DeclareMathOperator{\Adm}{Adm}
\DeclareMathOperator{\Conv}{Conv}
\DeclareMathOperator{\ext}{ext}
\DeclareMathOperator{\Perm}{Perm}
\DeclareMathOperator{\OGr}{OGr}

\newcommand{\ad}{\ensuremath{\mathrm{ad}}\xspace}
\newcommand{\aff}{\ensuremath{\mathrm{a}}\xspace}
\newcommand{\der}{\ensuremath{\mathrm{der}}\xspace}

\renewcommand{\int}{\ensuremath{\mathop{\mathrm{int}}}\xspace}
\newcommand{\loc}{\ensuremath{\mathrm{loc}}\xspace}
\newcommand{\naive}{\ensuremath{\mathrm{naive}}\xspace}

\newcommand{\Permsp}{\ensuremath{\Perm^\mathrm{sp}}\xspace}
\newcommand{\spin}{\ensuremath{\mathrm{spin}}\xspace}

\begin{document}

\renewcommand{\O}{\ensuremath{\mathscr{O}}\xspace}

\title[Topological flatness of orthogonal local models. I]{Topological flatness of orthogonal local models in the split, even case. I}
\author{Brian D. Smithling}
\address{University of Toronto, Department of Mathematics, 40 St.\ George St.,\ Toronto, ON  M5S 2E4, Canada}
\email{bds@math.toronto.edu}

\begin{abstract}
Local models are schemes, defined in terms of linear algebra, that were introduced by Rapoport and Zink to study the \'etale-local structure of integral models of certain PEL Shimura varieties over $p$-adic fields.  A basic requirement for the integral models, or equivalently for the local models, is that they be flat.  In the case of local models for even orthogonal groups, Genestier observed that the original definition of the local model does not yield a flat scheme.  In a recent article, Pappas and Rapoport introduced a new condition to the moduli problem defining the local model, the so-called spin condition, and conjectured that the resulting ``spin'' local model is flat.  We prove a weak form of their conjecture in the split, Iwahori case, namely that the spin local model is topologically flat.  An essential combinatorial ingredient is the equivalence of $\mu$-admissibility and $\mu$-permissibility for two minuscule cocharacters $\mu$ in root systems of type $D$.
\end{abstract}
\maketitle

\section{Introduction}
\numberwithin{equation}{section}

An important problem in the arithmetic theory of Shimura varieties is the definition and subsequent study of reasonable integral models.  For certain PEL Shimura varieties with parahoric level structure at $p$, Rapoport and Zink \cite{rapzink96} have constructed natural models over the ring of integers 
in the completion of the reflex field at any place lying over $p$.
One of the most basic requirements for the models is that they be flat.  The essential tool to investigate this and other questions of a local nature, also introduced in \cite{rapzink96}, is the \emph{local model:} this is a scheme \'etale-locally isomorphic to the original model, but defined in terms of a purely linear-algebraic moduli problem, and thus --- at least in principle --- more amenable to direct study.

Local models for groups involving only types $A$ and $C$ have received much study in the past decade; see, for example, work of Pappas \cite{pap00}, G\"ortz \cites{goertz01,goertz03,goertz04,goertz05}, Haines and Ng\^o \cite{hngo02b}, Pappas and Rapoport \cites{paprap03,paprap05,paprap08,paprap09}, Kr\"amer \cite{kr03}, and Arzdorf \cite{arz09}.  By contrast, the subject of this paper is the essential case of type $D$:  local models for the split orthogonal similitude group $GO_{2n}$ with Iwahori level structure.

Unfortunately, as observed by Genestier (see \cite{paprap09}*{\s8.3, p.~560}), the local model defined in \cite{rapzink96} \emph{fails to be flat} in the orthogonal case, even when the group is split; subsequently, this scheme has come to be renamed the \emph{naive local model $M^\naive$}.  Failure of flatness has also been observed, first by Pappas \cite{pap00}, for local models in type $A$ and $C$ cases for groups that split only after a ramified field extension.  As in these cases, there is a ``brute force'' correction available to non-flatness of $M^\naive$: one simply \emph{defines} the true local model $M^\loc$ to be the scheme-theoretic closure in $M^\naive$ of the generic fiber $M^\naive_\eta$.  A priori, this definition of $M^\loc$ carries the disadvantage of not admitting a ready moduli-theoretic description.  
Thus it is of interest when such a description can be found.  In \cite{paprap09} Pappas and Rapoport propose to describe $M^\loc$ by adding a new condition, the so-called \emph{spin condition} (see \s\ref{ss:spin_cond}), to the moduli problem defining $M^\naive$.
We denote by $M^\spin$ the subscheme of $M^\naive$ representing Pappas's and Rapoport's strengthened moduli problem.  One obtains $M^\spin \subset M^\naive$ as a closed subscheme, and Pappas and Rapoport show that the generic fibers of the two schemes agree.  They conjecture the following.

\begin{conj*}[Pappas-Rapoport \cite{paprap09}*{Conj.\ 8.1}]
$M^\spin = M^\loc$, that is, $M^\spin$ is the scheme-theoretic closure in $M^\naive$ of the generic fiber.
\end{conj*}

Although the conjecture remains open in general, Pappas and Rapoport have obtained
a considerable amount of computer evidence in support of it, and they explicitly work out the case $n=1$ and part of the case $n=2$ in \cite{paprap09}.  Hand calculations in the case $n = 3$ show that $M^\spin$ is indeed flat with reduced special fiber.  The main result of this paper is the following weak form of the conjecture.

\begin{thm*}[\ref{st:main_result}]
$M^\spin$ is topologically flat, that is, it has dense generic fiber.
\end{thm*}

In other words, the theorem asserts that the underlying topological spaces of $M^\spin$ and $M^\loc$ are the same.  The strategy to prove the theorem is the same as that pioneered in G\"ortz's original paper \cite{goertz01}:  we 
\begin{enumerate}
\item
   embed the special fiber $M^\naive_k$ in an appropriate affine flag variety \F, this time attached to $GO_{2n}$, over the residue field $k$;
\item\label{it:sch_cells}
   identify the set-theoretic images of $M^\naive_k$ and $M^\spin_k$ as unions of certain Schubert cells in \F, and obtain a good description of the Schubert cells occurring in the image of $M^\spin_k$; and
\item
   show that the Schubert cells in the image of $M^\spin_k$ are all in the closure of the generic fiber.
\end{enumerate}

By far, \eqref{it:sch_cells} is the most nontrivial.  The problem of obtaining a good description of the Schubert cells occurring in $M^\spin_k$ is essentially that of identifying the Schubert cells of maximal dimension in $M^\spin_k$, since these parametrize the irreducible components of $M^\spin_k$; and this translates to a purely combinatorial problem in the Iwahori-Weyl group $\wt W$ of $GO_{2n}$, which indexes the Schubert cells in \F.  In this form, the problem becomes essentially that of \emph{$\mu$-permissibility vs.\ $\mu$-admissibility} considered by Kottwitz-Rapoport \cite{kottrap00} and subsequently by Haines-Ng\^o \cite{hngo02a}.  More precisely, consider the dominant minuscule cocharacters
\[\tag{$1.1$}\label{disp:mu's}
   \mu_1 := (1^{(n)},0^{(n)})
   \quad\text{and}\quad
   \mu_2 := (1^{(n-1)},0,1,0^{(n-1)})
\]
for $GO_{2n}$ (expressed as cocharacters for the standard diagonal torus in the ambient
$GL_{2n}$), and regard them as translation elements in $\wt W$.  Let $W^\circ$ denote the finite Weyl group of the identity component $GO_{2n}^\circ$.  The special fiber $M_k^\spin$ has two connected components, and it is easy to see that the Schubert cells corresponding to the $W^\circ$-conjugates of $\mu_1$ (resp., $\mu_2$) are all contained in one component (resp., the other).  Of course, the closures of the Schubert cells obtained in this way are again contained in $M^\spin_k$.  For $\mu \in \{\mu_1,\mu_2\}$, the \emph{$\mu$-admissible set $\Adm^\circ(\mu)$} consists of the $w\in \wt W$ whose corresponding Schubert cell $C_w$ is contained in the closure of $C_{\mu'}$ for some $\mu' \in W^\circ \mu$.  On the other hand, the condition for a given Schubert cell $C_w$ to be contained in $M_k^\spin$ admits a combinatorial formulation in terms of $w$, and we define the \emph{$\mu$-spin-permissible set $\Permsp(\mu)$} to consist of the $w \in \wt W$ for which $C_w$ is contained in the connected component of $M_k^\spin$ marked by $\mu$.  There is also a third set to consider, the \emph{$\mu$-permissible set $\Perm(\mu)$} defined in \cite{kottrap00}.

\begin{thm*}[\ref{st:main_result}, \ref{st:perm_spin-perm}]
For $\mu \in \{\mu_1,\mu_2\}$, we have equalities of subsets of $\wt W$
\[
   \Adm^\circ(\mu) = \Permsp(\mu) = \Perm(\mu).
\]
\end{thm*}

The theorem is an analog of theorems for $GL_n$ and $GSp_{2n}$ obtained by Kottwitz and Rapoport \cite{kottrap00}*{3.5, 4.5, 12.4}.  It is especially worth comparing with the symplectic case.  Indeed, denote by $\wt W_{GL_{2n}}$ (resp., $\wt W_{GSp_{2n}}$) the Iwahori-Weyl group for $GL_{2n}$ (resp., $GSp_{2n}$).  Then $\wt W$ and $\wt W_{GSp_{2n}}$ become identified under these groups' respective natural embeddings into $\wt W_{GL_{2n}}$.  However, the relevant admissible and permissible sets in $\wt W$ and $\wt W_{GSp_{2n}}$ do not agree.  In the symplectic case, Kottwitz and Rapoport show these sets are obtained by intersecting $\wt W_{GSp_{2n}}$ with the relevant sets in $\wt W_{GL_{2n}}$, so that the theorem for $GSp_{2n}$ follows from the theorem for $GL_{2n}$.  But there seems to be no such royal road in the orthogonal case.  To prove our  theorem, we go back to Kottwitz's and Rapoport's original argument for $GL_n$ and adapt it to the orthogonal setting, where some new subtleties arise.

Kottwitz and Rapoport define $\mu$-admissibility and $\mu$-permissibility for any co\-character $\mu$ in any extended affine Weyl group attached to a root datum, and they show that $\mu$-admissibility always implies $\mu$-permissibility.  However, Haines and Ng\^o \cite{hngo02a}*{7.2} have shown that the reverse implication does not hold in general.  On the other hand, motivated by considerations arising from Shimura varieties, Rapoport \cite{rap05}*{\s3, p.~283} has raised the question of whether $\mu$-admissibility and $\mu$-permissibility agree for \emph{minuscule} cocharacters $\mu$, or even for sums of dominant minuscule cocharacters.  In the particular setting of this paper, $\wt W$ contains exactly three dominant minuscule cocharacters modulo the subgroup $\ZZ\cdot (1,\dotsc,1)$: $\mu_1$, $\mu_2$, and
%
\[
   \mu_3 := (1,0^{(2n-2)},-1).
\]
We give a proof of the equality $\Adm^\circ(\mu_3) = \Perm(\mu_3)$ in \cite{sm09b}.  Results of Kottwitz-Rapoport \cite{kottrap00}, this paper, and \cite{sm09b} combine to answer Rapoport's question in the affirmative for all minuscule $\mu$ in root data involving only types $A$, $B$, $C$, and $D$.  By contrast, we shall show in \cite{sm10a} that the answer to the more optimistic question, namely whether $\mu$-admissibility and $\mu$-permissibility are equivalent for $\mu$ a sum of dominant minuscule cocharacters, can be negative.

Somewhat surprisingly, Pappas and Rapoport have discovered that a version of the spin condition also turns up in their study of local models for ramified, quasi-split $GU_n$ \cite{paprap09}.  We shall show that the ``spin'' local models they define are topologically flat in \cites{sm09c,sm10a}.

For simplicity, in this paper we focus solely on the case of Iwahori level structure, and we ignore the explicit connection between local models and Shimura varieties --- although this is certainly
our main source of motivation to study local models.
We intend to take up the case of general parahoric level structure, as well as the connection to Shimura varieties, in a subsequent paper.



We now outline the contents of the paper.  In \s\ref{s:loc_mod} we review the definitions of orthogonal local models, both the naive version and the strengthened version incorporating the Pappas-Rapoport spin condition.  Sections \ref{s:G0_2n}--\ref{s:wtW} consist of some preparation of a group-theoretic nature for our subsequent discussion of the affine flag variety for $GO_{2n}$ over $k$.  In \s\ref{s:aff_flag_vty} we review the affine flag variety itself.  In \s\ref{s:embedding}, we embed the special fiber of the naive local model into the affine flag variety, and we use this to reduce the question of topological flatness for the spin model to the combinatorial identity $\Adm^\circ(\mu) = \Permsp(\mu)$ for $\mu \in \{\mu_1,\mu_2\}$.  In \s\ref{s:adm_perm_sets} we prove the identity $\Adm^\circ(\mu) = \Permsp(\mu)$, as well as the identity $\Permsp(\mu) = \Perm(\mu)$, for $\mu \in \{\mu_1,\mu_2\}$; this forms the technical heart of the paper.


\subsection*{Acknowledgments}  It is a pleasure to express my thanks to Ulrich G\"ortz and Michael Rapo\-port for their generosity of time, conversation, and advice in support of this project.  I am further indebted to Rapoport for introducing me to the subject.  I also thank Tom Haines and Eva Viehmann for helpful conversation; G\"ortz, Robert Kottwitz, and Rapoport for offering comments and suggestions on a preliminary draft of this paper; and the referee for offering further comments.  The bulk of the work presented here was conducted at the Max-Planck-Institut f\"ur Mathematik in Bonn, which I am pleased to acknowledge for its support and excellent working conditions.

\subsection*{Notation}
To maintain a certain uniformity of exposition, we work with respect to a fixed integer $n \geq 2$; the case $n = 1$ is handled completely in \cite{paprap09}*{Ex.\ 8.2}.
We work over a discretely valued, non-Archimedean field $F$ with ring of integers \O, uniformizer $\pi$, and residue field $k$, which we assume of characteristic not $2$.  We also employ an auxiliary discretely valued, non-Archimedean field $K$, this time supposed Henselian with valuation $\ord$, ring of integers $\O_K$, uniformizer $t$, and the same residue field $k$; eventually $K$ will be the field $k((t))$ of Laurent series over $k$.

We relate objects by writing $\iso$  for isomorphic, $\ciso$  for canonically isomorphic, and $=$  for equal.  The expression $(a^{(r)}, b^{(s)}, \dotsc)$ denotes the tuple with $a$ repeated $r$ times, followed by $b$ repeated $s$ times, and so on. Given an element $i \in \{1,\dotsc,2n\}$, we write $i^* := 2n+1-i \in \{1,\dotsc,2n\}$.

\section{Orthogonal local models}\label{s:loc_mod}
\numberwithin{equation}{subsection}
We begin by recalling the definition and some of the discussion of orthogonal local models from the paper of Pappas and Rapoport \cite{paprap09}*{\s8}.

\subsection{Lattices}
In this subsection we collect some notation and terminology on \O-lattices in the vector space $V:= F^{2n}$.

Let $e_1$, $e_2,\dotsc$, $e_{2n}$ denote the standard ordered basis in $V$.  We endow $V$ with the split symmetric $F$-bilinear form $h$ whose matrix with respect to the standard basis is
\begin{equation}\label{disp:std_form}
   \begin{pmatrix}
     &  &  1\\
     & \iddots\\
     1
   \end{pmatrix};
\end{equation}
that is, $h(e_i,e_j) = \delta_{i^*,j}$.  Given an \O-lattice $\Lambda \subset V$, we denote by $\wh \Lambda$ the \emph{$h$-dual of $\Lambda$},
\[
   \wh \Lambda := \{\, x\in V \mid h(\Lambda,x) \subset \O \,\}.
\]
Then $\wh \Lambda$ is an \O-lattice in $V$, and $h$ restricts to a perfect \O-bilinear pairing
\begin{equation}\label{disp:perf_pairing}
   \Lambda \times \wh \Lambda \to \O.
\end{equation}

Given a nonempty collection \L of lattices in $V$, we say that \L is
\begin{itemize}
\item 
   \emph{periodic} if $a\Lambda \in \L$ for all $\Lambda \in \L$ and $a\in F^\times$;
\item
   \emph{self-dual} if $\wh \Lambda \in \L$ for all $\Lambda \in \L$; and
\item
   a \emph{chain} if the lattices in \L are totally ordered under inclusion.
\end{itemize}
We say that a periodic lattice chain is \emph{complete} if all successive quotients are $k$-vector spaces of dimension $1$.

For $i = 2nk+r$ with $0 \leq r \leq 2n-1$, we define the \O-lattice
\begin{equation}\label{disp:Lambda_i}
   \Lambda_i := \sum_{j=1}^r\pi^{-k-1}\O e_j + \sum_{j=r+1}^{2n} \pi^{-k}\O e_j \subset V.
\end{equation}
Then $\wh\Lambda_i = \Lambda_{-i}$ for all $i$, and the $\Lambda_i$'s form a complete, periodic, self-dual lattice chain $\Lambda_\bullet$, which we call the \emph{standard chain},
\[
   \dotsb \subset \Lambda_{-2} \subset \Lambda_{-1} \subset \Lambda_0 \subset \Lambda_1 \subset \Lambda_2 \subset \dotsb.
\]

Let $f_i\colon \O^{2n} \to \O^{2n}$ multiply the $i$th standard basis element by $\pi$ and send all other standard basis elements to themselves.  Then there is a unique isomorphism of chains of \O-modules
\begin{equation}\label{disp:latt_triv}
   \vcenter{
   \xymatrix{
      \dotsb\,\, \ar@{^{(}->}[r]
         & \Lambda_0\, \ar@{^{(}->}[r] 
         & \Lambda_1\, \ar@{^{(}->}[r]
         & \,\dotsb\,\, \ar@{^{(}->}[r] 
         & \Lambda_{2n}\, \ar@{^{(}->}[r]
         & \,\dotsb\\
      \dotsb\, \ar[r]^{f_{2n}}
         & \O^{2n} \ar[u]_\sim \ar[r]^-{f_1}
         & \O^{2n} \ar[u]_-\sim \ar[r]^-{f_2}
         & \,\dotsb\, \ar[r]^-{f_{2n}}
         & \O^{2n} \ar[u]_\sim \ar[r]^-{f_1}
         & \,\dotsb
   }
   }
\end{equation}
such that the leftmost vertical arrow identifies the standard ordered basis of $\O^{2n}$ with the ordered \O-basis $e_1,\dotsc,e_{2n}$ of $\Lambda_0$.

\subsection{Naive local models}
In this subsection we recall the definition of naive local models from Rapoport's and Zink's book \cite{rapzink96} in the orthogonal case.  Given an \O-module $M$ and an \O-scheme $S$, we write $M_S$ for the quasi-coherent $\O_S$-module $M \tensor_\O \O_S$.

Let \L be a periodic self-dual lattice chain in $V$. The \emph{naive local model
$M_\L^\naive$ attached to $\L$} is the following contravariant functor on the
category of \O-schemes.  Given an \O-scheme $S$, an $S$-point in $M_\L^\naive$ consists of, up to an obvious notion of isomorphism,
\begin{itemize}
\item
   a functor
   \[
      \xymatrix@R=0ex{
         \L\vphantom{(\O_S)} \ar[r] & {}(\text{$\O_S$-modules})\\
         \Lambda\vphantom{\F_\Lambda} \ar@{|->}[r] & \F_\Lambda,
      }
   \]
   where \L is regarded as a category in the obvious way; together with
\item
   an injection $\F_\Lambda \inj \Lambda_S$ for each $\Lambda\in\L$, functorial in $\Lambda$;
\end{itemize}
satisfying, for all $\Lambda\in\L$,
\begin{enumerate}
\renewcommand{\theenumi}{LM\arabic{enumi}}
\item 
   $\F_\Lambda$ embeds in $\Lambda_S$ as an $\O_S$-locally direct summand of rank $n$;
\item\label{it:period_cond}
   the isomorphism $\Lambda_S \isoarrow (\pi \Lambda)_S$ obtained by tensoring $\Lambda
   \xra[\sim]\pi \pi \Lambda$ identifies $\F_\Lambda$ with $\F_{\pi \Lambda}$; and
\item\label{it:perp_cond}
   the perfect $\O_S$-bilinear pairing $\Lambda_S \times \wh \Lambda_S \to \O_S$
   obtained by tensoring \eqref{disp:perf_pairing} identifies $\F_\Lambda^\perp \subset \wh \Lambda_S$
   with $\F_{\wh \Lambda}$, where for any $\O_S$-submodule $M \subset \Lambda_S$, $M^\perp \subset \wh \Lambda_S$ is the subsheaf of sections that pair to $0$ with all sections of $M$.
\end{enumerate}

The functor $M_\L^\naive$ is plainly represented by a closed subscheme, which
we again denote $M_\L^\naive$, of a finite product of Grassmannians over
$\Spec \O$.

If $\pi$ is invertible on $S$, then any inclusion $\Lambda \subset \Lambda'$ of
\O-lattices becomes an isomorphism after tensoring with $\O_S$. Hence, for such
$S$, any $S$-point of $M_\L^\naive$ is determined by $\F_\Lambda \inj \Lambda_S$ for any single $\Lambda\in\L$. Hence $M_\L^\naive$ has generic fiber $\OGr(n,2n)_F$, the
orthogonal Grassmannian of totally isotropic $n$-planes in $2n$-space; this is a smooth $\binom n 2$-dimensional scheme with two
components, each isomorphic to $SO(h)/P$, where $P \subset SO(h)$ is a parabolic subgroup
stabilizing some totally isotropic $n$-plane.

In this paper we restrict to the Iwahori case, that is, to local models attached to \emph{complete} lattice chains.  It is not hard to verify directly that the special orthogonal group $SO(h)(F)$ acts transitively on the complete periodic self-dual lattice chains in $V$.  Hence the local models attached to any two complete lattice chains are isomorphic.  We shall work with respect to the standard chain $\Lambda_\bullet$, and we abbreviate $M^\naive := M_{\Lambda_\bullet}^\naive$.

The chain isomorphism \eqref{disp:latt_triv} permits a very concrete description of the points of $M^\naive$:  an $R$-point consists of $R$-submodules $\F_0$, $\F_1,\dotsc$, $\F_{2n} \subset R^{2n}$, each a locally direct summand of rank $n$, such that $(f_i \tensor R)(\F_{i-1}) \subset \F_i$ for all $i = 1,\dotsc$, $2n$; $\F_0 = \F_{2n}$; and $\F_i^\perp = \F_{2n-i}$ for all $i = 1,\dotsc,$ $2n$, where $R^{2n}$ carries the split symmetric form having matrix \eqref{disp:std_form} with respect to its standard basis.

\subsection{The spin condition of Pappas and Rapoport}\label{ss:spin_cond}
In \cite{paprap09}, Pappas and Rapoport introduce a conjectural correction to
the non-flatness of $M_\L^\naive$ by adding a new constraint, the \emph{spin
condition}, to the moduli problem. They define the spin condition
in the case of an arbitrary nondegenerate symmetric bilinear form $h$ on $V$. We
are only concerned in this paper with the case that $h$ is \emph{split}. The formulation
of the spin condition simplifies a bit in the split case: namely, we can
get by without explicit use of the \emph{discriminant algebra} of
\cite{paprap09}*{\s7.1}. It is a simple exercise to check that the formulation of
the spin condition we're about to give is equivalent to the spin condition in 
\cite{paprap09} in the split case.

To formulate the spin condition, we shall recall only the bare minimum of linear
algebra we need from \cite{paprap09}*{\s7}. In particular, we refer to
\cite{paprap09} for a more expansive and satisfying version of the
following discussion.

For a subset $E \subset \{1,\dotsc,2n\}$ of cardinality $n$, set
\begin{equation}\label{disp:e_E}
   e_E := e_{j_1}\wedge\dotsb\wedge e_{j_n} \in \sideset{}{_F^n}{\bigwedge} V,
\end{equation}
where $E = \{j_1,\dotsc,j_n\}$ with $j_1 < \dotsb < j_n$.  Given such $E$, we
also set 
\begin{equation}\label{disp:E^*_E^perp}
   E^* := 2n+1-E \qquad\text{and}\qquad E^\perp := (E^*)^c = (E^c)^*,
\end{equation}
where the set complements are taken in $\{1,\dotsc,2n\}$.  Then $E^*$ specifies the indices $j'$ such that $h(e_j,e_{j'}) = 1$ for some $j\in E$, and $E^\perp$ specifies the $j'$ such that $h(e_j,e_{j'}) = 0$ for all $j\in E$.

We define an operator $a$ on $\bigwedge^n V$ by its action on the standard
basis elements $e_E$ for varying $E$,
\[
   a(e_E) := \sgn(\sigma_E)e_{E^\perp},
\]
where $\sigma_E$ is the permutation on $\{1,\dotsc,2n\}$ sending $\{1,\dotsc,n\}$ to the elements of $E^*$ in decreasing order, and sending $\{n+1,\dotsc,2n\}$ to the elements of $E^\perp$ in increasing order.  Then $a$ satisfies $a^2 = \id_{\bigwedge^n V}$ \cite{paprap09}*{Prop.\ 7.1}.  Hence $\bigwedge^n V$ decomposes as
\[
   \sideset{}{^n}{\bigwedge} V 
      = \Bigl(\sideset{}{^n}{\bigwedge} V\Bigr)_+ 
            \oplus \Bigl(\sideset{}{^n}{\bigwedge} V\Bigr)_-,
\]
where $\bigl(\bigwedge^n V\bigr)_\pm$ denotes the $\pm 1$ eigenspace for $a$.  Using that $a^2$ is the identity, we see that
\begin{equation}\label{disp:spin_basis}
   \Bigl(\sideset{}{^n}{\bigwedge} V\Bigr)_\pm 
      = \spn_F\{e_E \pm \sgn(\sigma_E)e_{E^\perp}\},
\end{equation}
where $E$ ranges through the subsets of $\{1,\dotsc,2n\}$ of cardinality $n$.

Now let $\Lambda \subset V$ be an \O-lattice.  Then $\bigwedge^n_\O \Lambda$ is naturally an \O-submodule of $\bigwedge^n_F V$, and we set
\[
   \Bigl(\sideset{}{_\O^n}{\bigwedge} \Lambda\Bigr)_\pm
      := \Bigl(\sideset{}{_\O^n}{\bigwedge} \Lambda \Bigr)
            \cap \Bigl(\sideset{}{_F^n}{\bigwedge} V\Bigr)_\pm.
\]

We are now ready to state the spin condition. Let \L be a periodic self-dual
lattice chain. We say that an $S$-point $\{\F_\Lambda \inj \Lambda_S\}_{\Lambda\in\L}$ of
$M_\L^\naive$ \emph{satisfies the spin condition} if
\begin{enumerate}
\renewcommand{\theenumi}{LM\arabic{enumi}}
\setcounter{enumi}{3}
\item
   Zariski locally on $S$, either $\bigwedge^n_{\O_S} \F_\Lambda$ is contained in
   \[
      \im\Bigl[\Bigl(\sideset{}{_\O^n}{\bigwedge} \Lambda\Bigr)_+\tensor_\O \O_S 
         \to \sideset{}{_{\O_S}^n}{\bigwedge}\Lambda_S \Bigr]
   \]
   for all $\Lambda$ in \L, or in 
   \[
      \im\Bigl[\Bigl(\sideset{}{_\O^n}{\bigwedge} \Lambda\Bigr)_-\tensor_\O \O_S 
         \to \sideset{}{_{\O_S}^n}{\bigwedge}\Lambda_S \Bigr]
   \]
   for all $\Lambda$ in \L.
\end{enumerate}
The \emph{spin local model attached to $\L$}, which we denote $M_\L^\spin$, is the closed subscheme of $M_\L^\naive$ whose points satisfy the spin condition.  Pappas and Rapoport show in \cite{paprap09}*{\s8.2.1} that the arrow $(M_\L^\spin)_F \inj (M_\L^\naive)_F$ on generic fibers is an isomorphism.

As in the previous subsection, when working with complete periodic self-dual chains \L, $M_\L^\spin$ is independent of \L up to isomorphism, and we put $M^\spin := M_{\Lambda_\bullet}^\spin$.

%
%
%
%

\section{Orthogonal similitude group}\label{s:G0_2n}
In this section we review some basic facts about split $GO_{2n}$.  We switch to working over the field $K$.  Except in \s\ref{ss:max_torus}, $K$ may be an arbitrary field of characteristic not $2$; in \s\ref{s:Iwahori} we'll return to our blanket assumptions on $K$ stated in the introduction.

\subsection{Orthogonal similitudes}
Abusing notation, we denote again by $h$ the symmetric bilinear form on $K^{2n}$ whose matrix with respect to the standard ordered basis is \eqref{disp:std_form}.  We denote by $G := GO_{2n} := GO(h)$ the algebraic group over $K$ of \emph{orthogonal
similitudes} of $h$: for any $K$-algebra $R$, $G(R)$ is the set of
elements $g\in GL_{2n}(R)$ satisfying $h_R(gx,gy) =
c(g)h_R(x,y)$ for some $c(g) \in R^\times$ and all $x$, $y\in R^{2n}$,
where $h_R$ is the induced form on $R^{2n}$.  As the form $h$
is nonzero, the scalar $c(g)$ is uniquely determined, and $c$ defines an
exact sequence of $K$-groups
\[
   1 \to O \to G \xra c \GG_m \to 1
\]
with evident kernel $O:= O_{2n} := O(h)$ the orthogonal group of $h$.  
The displayed sequence splits (noncanonically), so that the choice of a
splitting presents $G$ as a semidirect product $O \rtimes \GG_m$.

\subsection{Center}
The \emph{center} $Z := Z_G$ of $G$ consists of the scalar matrices; on
$R$-valued points,
\[
   Z(R) = \{\, r\cdot\id \in GL_{2n}(R) \mid r\in R^{\times} \,\},
\]
where $\id$ denotes the identity matrix, so that $Z \ciso \GG_m$.

We write $G_\ad := PGO_{2n} := PGO(h) := G/Z$ for the adjoint group.

\subsection{Connected components}
The group $G$ possesses two connected components.  For $g \in G(R)$ with $\Spec R$
connected, the corresponding morphism $\Spec R \to G$ factors through the
identity component or the non-identity component according as
$c(g)^{-n}\det(g)$ is $+1$ or $-1$, respectively.   The identity component
$G^\circ$ is split reductive.

\subsection{Standard maximal torus}\label{ss:max_torus}
Let $T$ denote the standard split maximal torus of diagonal matrices in $G$; on
$R$-points,
\[
   T(R):= \{\, \diag(a_1,\dotsc,a_{2n}) \in GL_{2n}(R) \mid 
	         a_1a_{2n} = a_2 a_{2n-1} = \dotsb = a_n a_{n+1} \,\},
\]
so that $T \iso \GG_m^{n+1}$.  

Now let us assume $K$ is as in the introduction, so that it is discretely valued with valuation ring $\O_K$ and uniformizer $t$.  Then we identify the cocharacter lattice $X_*(T)$ with $T(K)/T(\O_K)$ via the rule $\lambda \mapsto \lambda(t) \bmod T(\O_K)$, and we identify $T(K)/T(\O_K)$ with 
\[\tag{$*$}\label{disp:cochar}
   \{\, (r_1,\dotsc,r_{2n}) \in \ZZ^{2n} \mid r_1 + r_{2n} = \dots = r_n + r_{n+1} \,\}
\]
via $\ord$.  These identifications in turn identify
\begin{itemize}
   \item
      $X_*(T_\der)$ with the subgroup of \eqref{disp:cochar} of elements $(r_1,\dotsc,r_{2n})$ such that $r_1 + r_{2n} = \dots = r_n + r_{n+1} = 0$, where $G_\der := SO := SO_{2n} := SO(h)$ is the derived group of $G$ and $T_\der$ is its split maximal torus $T \cap G_\der$; and
   \item
      $X_*(T_\ad)$ with the quotient of \eqref{disp:cochar} by the subgroup $\ZZ\cdot (1,\dotsc,1)$, where $T_\ad := T/Z$ is the split maximal torus in $G_\ad$ obtained as the image of $T$.
\end{itemize}


\subsection{Roots, coroots}\label{ss:roots}
Let $\chi_i$ denote the character on $T$ sending
\[
   \diag(a_1,\dotsc,a_{2n}) \mapsto a_i.
\]
The \emph{roots} of the pair $(G, T)$ are the set
\begin{align*}
   \Phi_G := \Phi_{G,T} &:= \{\,\pm(\chi_i-\chi_j) \mid 1 \leq i < j < i^*\,\}\\
         &\phantom{:}= \{\, \pm(\chi_i - \chi_j),\pm(\chi_i + \chi_j - c) 
	         \mid 1 \leq i < j \leq n \,\},
\end{align*}
where we use the same symbol $c$ to denote the composite $T \inj G \xra c \GG_m$.  Of course, the roots of $G$ descend to the roots $\Phi_{G_\ad} := \Phi_{G_\ad, T_\ad}$ of the pair $(G_\ad, T_\ad)$.  When $n=1$, $G^\circ \iso \GG_m^2$ is abelian and $\Phi_G = \Phi_{G_\ad} = \emptyset$. Otherwise, the root system $\Phi_{G_\ad}$ is of type $A_1 \times A_1$ for $n =2$, $A_3$ for $n = 3$, and $D_n$ for $n \geq 4$.

For $n > 1$, we take the $n$ roots
\begin{equation}\label{disp:pos_coroots}
   \chi_1-\chi_2,\dotsc,\ \chi_{n-1}-\chi_n,\ \chi_{n-1} + \chi_n -c
\end{equation}
as simple roots.

For $1 \leq i \leq n$, let $\lambda_i \in X_*(T)$ denote the cocharacter
\[
   x \mapsto \diag(1,\dotsc,1,x,1,\dotsc,1,x^{-1},1,\dotsc,1),
\]
where $x$ is in the $i$th slot and $x^{-1}$ is in the $i^*$th slot.  Then for $n>1$, the \emph{coroots} consist of the cocharacters
\[
   (\chi_i - \chi_j)^\vee = \lambda_i - \lambda_j
   \quad\text{and}\quad
   (\chi_i + \chi_j -c)^\vee = \lambda_i + \lambda_j
\]
for $1 \leq i < j \leq n$.

\subsection{Weyl group}\label{ss:weyl_group}
The torus $T$ has \emph{normalizer} $N := N_G T$ in $G$ the algebraic group of
monomial matrices contained in $G$, and \emph{finite Weyl group}
\[
   W := W_{G,T} := N(K)/ T(K).
\]
The Weyl group $W$ acts naturally on the set of lines in $K^{2n}$ spanned by the standard ordered basis vectors, and this canonically identifies $W$ with the group $S^h_{2n}$ of permutations $\sigma$ of
$\{1,\dotsc,2n\}$ satisfying
\[
   \sigma(i^*) = \sigma(i)^* \quad \text{for all} \quad i.
\]
The group $S_{2n}^h$ decomposes as a semidirect product $\{\ZZ/2\ZZ\}^n \rtimes
S_n$, where the nontrivial element in the $i$th copy of $\ZZ/2\ZZ$ acts as the
transposition $(i,i^*)$, and where the symmetric group $S_n$ acts on
$\{1,\dotsc, n\}$ in the standard way and on $\{n+1,\dotsc,2n\}$ in the way
compatible with the display.

Note that $W$ is not the Weyl group attached to the root system
$\Phi_G$.  Rather, let
\begin{equation}\label{disp:W^circ}
   W^\circ := W_{G^\circ, T} := N_{G^\circ(F)}T(K)/T(K)
\end{equation}
denote the finite Weyl group of $T$ in $G^\circ$. Then $W^\circ \ciso W(\Phi_G)$ is naturally contained in $W$ as a subgroup of index $2$. In terms of permutations, $W^\circ$ corresponds to the elements of $S^h_{2n}$ which are \emph{even} as elements of the symmetric group $S_{2n}$.


\subsection{Fundamental group}\label{ss:fund_gp}
In terms of the identifications in \s\ref{ss:max_torus}, the \emph{coroot lattice}
\begin{equation}\label{disp:Q^vee}
   Q^\vee := Q^\vee(G,T) \subset X_*(T_\der) \subset X_*(T)
\end{equation}
consists of all $(r_1,\dotsc,r_{2n}) \in \ZZ^{2n}$ such that $r_1 + r_{2n} = \dots = r_n + r_{n+1} = 0$ and $r_1 + \dotsb + r_n$ is even.  The \emph{fundamental group} of $G$ is the fundamental group of the identity component $G^\circ$,
\[
   \pi_1(G) := \pi_1(G^\circ) := X_*(T)/Q^\vee \iso
   \ZZ/2\ZZ \oplus \ZZ.
\]
Note that the derived group $G_\der = G^\circ_\der = SO$ is not simply connected, as its fundamental group $X_*(T_\der)/Q^\vee \ciso \ZZ/2\ZZ$.

\section{Iwahori subgroup}\label{s:Iwahori}
We return to our assumptions on $K$ stated in the introduction.  In this section we discuss the standard Iwahori subgroup of $G(K)$.  In particular, we realize it as a lattice chain stabilizer. 

\subsection{Standard apartment}
Let $\B := \B(G_\ad)$ denote the building of $G_\ad$.  We call the apartment in \B associated with $T_\ad$ the \emph{standard apartment}, and we denote it by $\aaa := \aaa_{T_\ad}$.  In terms of the identifications in \s\ref{ss:max_torus},
\[
   \aaa = X_*(T_\ad) \tensor_\ZZ \RR 
      \ciso \frac{\{\, (r_1,\dotsc,r_{2n}) \in \RR^{2n} \mid r_1 + r_{2n} = \dots = r_n + r_{n+1} \,\}}{\RR \cdot (1,\dotsc,1)}.
\]

\subsection{Base alcove}
We take as our base alcove the alcove $A$ in \aaa containing the origin and contained in the \emph{negative} Weyl chamber relative to our choice of simple roots \eqref{disp:pos_coroots}.  The alcove $A$ has $n+1$ vertices
\begin{align*}
   a_0 &:= (0,\dotsc,0),\\
   a_{0'} &:= (-1,0^{(2n-2)},1),\\
   a_i &:= \bigl((-\tfrac 1 2)^{(i)},0^{(2n-2i)},(\tfrac 1 2)^{(i)}\bigr),\quad 2 \leq i \leq n-2,\\
   a_n &:= \bigl((-\tfrac 1 2)^{(n)},(\tfrac 1 2)^{(n)}\bigr),\\
   a_{n'} &:= \bigl((-\tfrac 1 2)^{(n-1)},\tfrac 1 2,-\tfrac 1 2,(\tfrac 1 2)^{(n-1)}\bigr),
\end{align*}
all taken mod $\RR\cdot (1,\dotsc,1)$.  The vertices $a_0$, $a_{0'}$, $a_n$, and $a_{n'}$ are hyperspecial; the other vertices are nonspecial.

\subsection{Standard Iwahori subgroup}  Let us say that an \emph{Iwahori subgroup of $G(K)$} is just an Iwahori subgroup of $G^\circ(K)$ in the usual sense for any connected reductive group.  We denote by $B$ the Iwahori subgroup of $G(K)$ attached to our base alcove $A$, and we call it the \emph{standard Iwahori subgroup}.

To realize $B$ as a lattice chain stabilizer, let $\lambda_\bullet$ denote the $\O_K$-lattice chain in $K^{2n}$ defined as the obvious analog of the \O-lattice chain $\Lambda_\bullet$ \eqref{disp:Lambda_i}, where $\O_K$ replaces \O and $t$ replaces $\pi$.  Let
\[
   P_{\lambda_\bullet} := \{\, g \in G(K) \mid g\lambda_i = \lambda_i \text{ for all $i$}\,\}.
\]
Then $P_{\lambda_\bullet}$ is the intersection of $G(K)$ with the standard Iwahori subgroup
\[
   \begin{pmatrix}
      \O_K^\times & & \O_K\\
        & \ddots\\
      t \O_K & & \O_K^\times
   \end{pmatrix}
\]
of $GL_{2n}(K)$.

\begin{prop}\label{st:B=P}
$B = P_{\lambda_\bullet}$.
\end{prop}

To prepare for the proof, recall \citelist{\cite{hrap08}*{3}\cite{btII}*{remark after 5.2.8}} that for any facet $F$ in \B, the associated parahoric subgroup $P_F$ is precisely the set of all $g \in G^\circ (K)$ with trivial Kottwitz invariant such that $ga = a$ for all vertices $a$ of $F$.  The Kottwitz homomorphism admits a simple description for any split connected reductive group $H$ with split maximal torus $S$:  it is a functorial surjective homomorphism
\[
   \kappa_H\colon H(K) \surj \pi_1(H)
\]
which is characterized in terms of the Cartan decomposition
\[
   H(K) = H(\O_K)S(K)H(\O_K)
\]
as being trivial on $H(\O_K)$ and as restricting on $S(K)$ to the composite
\[
   S(K) \surj S(K)/S(\O_K) \ciso X_*(S) \surj X_*(S)/Q_{H,S}^\vee = \pi_1(H),
\]
where $Q^\vee_{H,S}$ denotes the coroot lattice for $S$ in $H$.  In the case of our group $G^\circ$, upon choosing a splitting $G^\circ \iso SO \rtimes \GG_m$, we identify its fundamental group with $\ZZ/2\ZZ \oplus \ZZ$ in the way that $\kappa_{G^\circ}$ sends $(g,x) \mapsto \bigl(\mathop{\kappa_{SO}} (g), \ord (x)\bigr)$.

\begin{lem}
$P_{\lambda_\bullet} \subset G^\circ (K)$.
\end{lem}

\begin{proof}
Given $g\in P_{\lambda_\bullet}$, we must show that $c(g)^n = \det(g)$.  Since $\charac (k) \neq 2$ and the only other
possibility is $c(g)^n = -\det(g)$, it suffices to show $c(g)^n \equiv \det(g)
\bmod t$.

Write $g$ as a matrix $(g_{ij})$.  Since $g$ preserves the form $h$ up
to the scalar $c(g)$, the $i$th and $i^*$th columns of $g$ pair to $c(g)$
for $1 \leq i \leq 2n$.  Hence
\[
   c(g) \equiv g_{i,i}g_{i^*,i^*} \mod t .
\]
Hence 
\[
   c(g)^n \equiv \prod_{i=1}^{2n} g_{i,i} \equiv \det(g) \mod t.\qedhere
\]
\end{proof}

\begin{proof}[Proof of \eqref{st:B=P}]
Since plainly $P_{\lambda_\bullet} \subset G(\O_K)$, the lemma implies $P_{\lambda_\bullet} \subset \ker \kappa_{G^\circ}$.  On the other hand, we see from the explicit form of $\kappa_{G^\circ}$ that any $g\in B \subset \ker\kappa_{G^\circ}$ has determinant of valuation $0$.  The equality $B = P_{\lambda_\bullet}$ now follows easily from the explicit expressions for the vertices of $A$ and from the usual identification of \B with homothety classes of certain norms on $K^{2n}$.
\end{proof}

\section{Iwahori-Weyl group}\label{s:wtW}
In this section we discuss a few matters related to the Iwahori-Weyl group of $G$.  Once we specialize to the function field case later on, we'll use the Iwahori-Weyl group to index Schubert cells in the affine flag variety attached to $G$.

\subsection{Iwahori-Weyl group}
The \emph{Iwahori-Weyl group $\wt W$} of $G$ is the group
\[
   \wt W := \wt W_G := \wt W_{G,T} := N(K)/T(\O_K).
\]
We shall also need the Iwahori-Weyl group $\wt W^\circ$ of the identity component $G^\circ$,
\[
   \wt W^\circ := \wt W_{G^\circ} := \wt W_{G^\circ,T} := N_{G^\circ(K)}T(K)/T(\O_K).
\]
It will be convenient for us to single out the permutation matrix $\tau\in G(K)$ corresponding to the transposition $(n, n+1)$.  Then $\tau$ is contained in the non-identity component of $O(K)$ and normalizes $T$, so that there is a decomposition $\wt W = \wt W^\circ \amalg \tau\wt W^\circ$.\label{ss:tau}

\subsection{Affine Bruhat decomposition} 
Let $H$ be a split connected reductive $K$-group with split maximal torus $S$, and
let $I \subset H(K)$ be the Iwahori subgroup corresponding to an alcove in the apartment associated with $S$.  The \emph{affine Bruhat decomposition}
asserts that the natural map $\wt W _{H,S} := N_{H(K)}S(K)/S(\O_K) \to I \bs H(K) /
I$ sending $n \bmod S(\O_K) \mapsto I n I$ is a bijection; see Haines and Rapoport \cite{hrap08}*{8}.
In this subsection we show that the analogous result still holds for our disconnected group $G$.


\begin{prop}\label{st:wtW=double_cosets}
The natural map $\wt W \to B \bs G(K)/B$ is a bijection of sets.
\end{prop}


\begin{proof}
This follows from the affine Bruhat decomposition for $G^\circ$.  Indeed, we have decompositions
\[
   \wt W = W^\circ \amalg \tau \wt W^\circ
\]
and
\[
   B \bs G(K) /B = \bigl(B\bs G^\circ(K) / B\bigr) \amalg \bigl(B \bs \tau G^\circ(K) / B\bigr),
\]
and we at least obtain $W^\circ \isoarrow B\bs G^\circ(K) / B$.  So it remains to show that the map $\tau\wt W^\circ \to B\bs \tau G^\circ(K) / B$ is a bijection.  Since $\tau$ plainly stabilizes the base alcove $A$, $\tau$ normalizes $B$.  Hence $B\bs \tau G^\circ(K) / B = \tau\bigl(B\bs G^\circ(K) / B\bigr)$. So we get what we need again from the affine Bruhat decomposition for $G^\circ$.
\end{proof}


\subsection{Semidirect product decompositions}\label{ss:sdirect_prods}
As usual, $\wt W$ admits two standard semidirect product decompositions, which
we now describe.

The first decomposition is 
\[
   \wt W \ciso \bigl(T(K)/T(\O_K)\bigr) \rtimes W \ciso X_*(T) \rtimes W,
\]
where we lift the finite Weyl group $W$ to $N(K)$ by choosing permutation matrices as representatives, and where we identify $T(K)/T(\O_K) \ciso X_*(T)$ as in \s\ref{ss:max_torus}.  In this way, we refer to $X_*(T)$ as the \emph{translation subgroup} of $\wt W$, and we denote the image of $\mu\in X_*(T)$ in $\wt W$ by $t_\mu$.  Concretely, in terms of our identifications for $X_*(T)$ and $W$ in \s\ref{ss:max_torus} and \s\ref{ss:weyl_group}, respectively, we have
\begin{equation}\label{disp:wtW_sdprod}
   \wt W \ciso \{\, (r_1,\dotsc,r_{2n}) \in \ZZ^{2n} \mid r_1 + r_{2n} = \dots = r_n + r_{n+1} \,\} \rtimes S^h_{2n}.
\end{equation}

The second decomposition involves the \emph{affine Weyl group $W_\aff$} of $G$. In terms of our first semidirect product decomposition, we have $W_\aff := Q^\vee \rtimes W^\circ \subset \wt W$, where we recall $Q^\vee \subset X_*(T)$ is the coroot lattice \eqref{disp:Q^vee} and $W^\circ \subset W$ is the finite Weyl group of $G^\circ$ \eqref{disp:W^circ}.  Then
\begin{itemize}
\item
   $W_\aff$ is a normal subgroup of $\wt W$; and
\item
   $W_\aff$ is canonically identified with the affine Weyl group of the root system $(\Phi_{G_\ad}, X_*(T_\ad)\tensor \RR)$, so that $W_\aff$ acts simply transitively on the set of alcoves in the standard apartment.
\end{itemize}
Hence $\wt W$ is the semidirect product of $W_\aff$ and the stabilizer $\Omega$
of the base alcove $A$,
\[
   \wt W \ciso W_\aff \rtimes \Omega.
\]


We remark that, in contrast with the analogous situation for a connected reductive group, the quotient $\wt W/W_\aff \ciso \Omega$ is \emph{nonabelian}.  Indeed, we have an identification $\wt W/W_\aff \ciso X_*(T)/Q^\vee \rtimes W/W^\circ$; and the point is that $W/W^\circ$ is nontrivial and  acts nontrivially on $X_*(T)/Q^\vee$.  To see this, recall the cocharacters $\mu_1$, $\mu_2 \in X_*(T)$ from \eqref{disp:mu's}; these yield distinct dominant minuscule coweights for $G_\ad$.  Hence
$\mu_1$ and $\mu_2$ have distinct images in $X_*(T)/Q^\vee$.  But $W/W^\circ
\ciso \ZZ/2\ZZ$ is generated by the image of $\tau$, and the action of $\tau$ on $X_*(T)$ interchanges $\mu_1$ and $\mu_2$.

\subsection{Length, Bruhat order}\label{ss:bo}
The decomposition $\wt W \ciso W_\aff \rtimes \Omega$ furnishes a \emph{length function} and \emph{Bruhat order} on $\wt W$ in the standard way, which we briefly recall. The reflections through the walls of the base alcove form a generating set for the Coxeter group $W_\aff$. Hence we get a length function $\ell$ and Bruhat order $\leq$ on $W_\aff$. These then extend to $\wt W$ as usual: for $x\omega$, $x'\omega' \in \wt W$ with $x$, $x'\in W_\aff$ and $\omega$, $\omega'\in\Omega$, we have $\ell(x\omega) := \ell(x)$ and $x\omega \leq x'\omega'$ exactly when $\omega = \omega'$ and $x\leq x'$ in $W_\aff$.

We remark now that, in the function field case, the Bruhat order gives the correct closure relations for Schubert varieties in the affine flag variety; see \eqref{st:BO_closure_relns} below.


\subsection{\texorpdfstring{$\mu$}{mu}-admissible set}\label{ss:mu_adm}
Let $\mu \in X_*(T)$ be a cocharacter.  Then we define the \emph{$\mu$-admissible set $\Adm(\mu) \subset \wt W$} in the most obvious way based on the usual definition for connected groups,
\[
   \Adm(\mu) := \{\, w\in\wt W \mid w \leq \sigma t_\mu \sigma^{-1} 
	                 \text{ for some } \sigma \in W \,\}.
\]
Of course, we in fact have $\Adm(\mu) \subset \wt W^\circ$.

In the case of a connected group, all elements of the $\mu$-admissible set are congruent mod $W_\aff$ since $\wt W/ W_\aff$ is abelian; in fact, as shown by Rapoport \cite{rap05}*{3.1}, this common element in $\wt W/ W_\aff$ depends only on the geometric conjugacy class of $\mu$.  In the case of our disconnected group $G$, we have already seen that $\wt W/W_\aff$ is nonabelian.  And indeed, it can happen that $\Adm(\mu)$ possesses elements that are \emph{distinct} mod $W_\aff$.  For example, this is the case for $\mu = \mu_1$ \eqref{disp:mu's}, since $\tau\mu_1\tau^{-1} = \mu_2$.

To make this a bit more precise, consider
\[
   \Adm^\circ(\mu) := \{\, w\in\wt W \mid w \leq \sigma t_\mu \sigma^{-1} 
	                 \text{ for some } \sigma \in W^\circ \,\},
\]
the \emph{admissible set of $\mu$ in $G^\circ$}.  Then for any $\mu$,
\[
   \Adm(\mu) = \Adm^\circ(\mu) \cup \Adm^\circ(\tau\mu\tau^{-1}).
\]
Hence the study of admissible sets for $G$ reduces to the study of admissible
sets for $G^\circ$.  We see from this last display that $\Adm(\mu)$ contains
either $1$ or $2$ elements mod $W_\aff$; the union is disjoint precisely in the
latter case.

\subsection{Extended alcoves}\label{ss:extd_alcs}
We conclude the section by giving a combinatorial description of $\wt W$ in terms of \emph{extended alcoves} that will be convenient later on when we consider Schubert cells in the affine flag variety.  Identifying $\wt W \ciso \wt W_{GSp_{2n}}$ as in \s\ref{ss:bo}, our description will be the same as that for $\wt W_{GSp_{2n}}$ given by Kottwitz and Rapoport in \cite{kottrap00}*{4.2}, except we shall adopt some slightly different conventions to make the relation with the affine flag variety clearer.  Following the notation of \cite{kottrap00}, given $v \in \ZZ^{2n}$, we write $v(i)$ for the $i$th entry of $v$, and we write $\Sigma v$ for the sum of the entries of $v$.  We write $v \geq w$ if $v(i) \geq w(i)$ for all $i$.

An \emph{extended alcove} for $G$ is a sequence $v_0,\dotsc$, $v_{2n-1}$ of elements in $\ZZ^{2n}$ such that, putting $v_{2n} := v_0 - (1,\dotsc,1)$,
\begin{enumerate}
\renewcommand{\theenumi}{A\arabic{enumi}}
\item
   $v_0 \geq v_1 \geq \dotsb \geq v_{2n}$;
\item
   $\Sigma v_i = \Sigma v_{i-1} - 1$ for all $1 \leq i \leq 2n$; and
\item\label{it:dual_cond}
   there exists $d\in \ZZ$ such that $v_{i}(j) + v_{2n-i}(j^*) = d$ for all $1 \leq i,j \leq 2n$.
\end{enumerate}
We frequently refer to \eqref{it:dual_cond} as the \emph{duality condition}.  The sequence of elements $\omega_i := \bigl((-1)^{(i)},0^{(2n-i)}\bigr)$ is an extended alcove, with $d = -1$, which we call the \emph{standard extended alcove}.  The group $\wt W$ acts naturally on extended alcoves via its expression in \eqref{disp:wtW_sdprod}.  Just as in \cite{kottrap00}*{4.2}, this action is simply transitive, and we identify $\wt W$ with the set of extended alcoves by taking the standard extended alcove as base point.

\section{Affine flag variety}\label{s:aff_flag_vty}
In this section we discuss a few basic aspects of the affine flag variety attached to $G$ in the function field case.  We take $K = k((t))$ and $\O_K = k[[t]]$ from now on.  We follow closely \cite{paprap09}*{\s\s3.1--3.2}.

\subsection{Affine flag variety}
We recall the construction of the affine flag variety over
$k$.

To begin, the \emph{loop group} $LG$ is the functor on $k$-algebras
\[
   LG\colon R \mapsto G\bigl(R((t))\bigr),
\]
where $R((t))$ is the ring of Laurent series with coefficients in $R$, regarded
as a $K$-algebra in the obvious way.

Next recall the standard Iwahori subgroup $B \subset G(K)$.  Abusing notation, we denote again by $B$ the associated Bruhat-Tits scheme over $\O_K$; this is a smooth affine group scheme with generic fiber $G^\circ$ and with connected special fiber.  We denote by $L^+B$ the functor on $k$-algebras
\[
   L^+B \colon R \mapsto B\bigl(R[[t]]\bigr),
\]
where $R[[t]]$ is regarded as an $\O_K$-algebra in the obvious way.

Finally, the \emph{affine flag variety $\F$} is the fpqc quotient $LG/L^+B$ of
sheaves on the category of $k$-algebras. It is an ind-$k$-scheme of ind-finite
type \cite{paprap08}*{1.4}.  Note that \F is a disjoint union of two copies of the affine flag variety $\F^\circ := LG^\circ/L^+B$ for $G^\circ$,
\[
   \F = \F^\circ \amalg \tau \F^\circ,
\]
with $\tau \in G(K)$ the element of \s\ref{ss:tau}.

\subsection{Lattice-theoretic description}\label{ss:latt_desc}
In this subsection we describe points on the affine flag variety in terms of certain lattice chains in $K^{2n}$.  Let $R$ be a $k$-algebra.  Recall that an \emph{$R[[t]]$-lattice in $R((t))^{2n}$} is an $R[[t]]$-submodule $L \subset R((t))^{2n}$ which is free as an $R[[t]]$-module Zariski-locally on $\Spec R$, and such that the natural arrow $L \tensor_{R[[t]]} R((t)) \to R((t))^{2n}$ is an isomorphism.  Borrowing our earlier notation, given an $R[[t]]$-lattice $L$, we write $\wh L$ for the dual lattice
\[
   \wh L := \{\, x \in R((t))^{2n} \mid h_{R((t))}(L,x) \subset R[[t]] \, \},
\]
where $h_{R((t))} := h \tensor_K R((t))$ is the induced form on $R((t))^{2n}$.   We say that an indexed sequence
\[
   \dotsb \subset L_{-1} \subset L_0 \subset L_1 \subset \dotsb
\]
of lattices in $R((t))^{2n}$ is an \emph{indexed chain} if all successive quotients are locally free $R$-modules.  We say that an indexed chain $L_\bullet$ is \emph{periodic} if $t L_i = L_{i-2n}$ for all $i$, and \emph{complete} if all successive quotients are locally free $R$-modules of rank $1$.

We define $\F'$ to be the functor on the category of $k$-algebras that sends each algebra $R$ to the set of all complete periodic indexed lattice chains $L_\bullet$ in $R((t))^{2n}$ with the property that Zariski-locally on $\Spec R$, there exists a scalar $\alpha \in R((t))^\times$ such that $\wh L_i = \alpha L_{-i}$ for all $i$.  The natural action of $G\bigl(R((t))\bigr)$ on $R((t))^{2n}$ yields an action of $LG$ on $\F'$.  Taking the standard chain $\lambda_\bullet \in \F'(k)$ as base point, we obtain a map $LG \to \F'$ which induces, quite as in  \cite{paprap09}*{\s3.2},%
\footnote{Though note that the scalar $\alpha$ in the definition of $\F_I$ in \cite{paprap09} should only be required to exist Zariski-locally, so that $\F_I$ satisfies the sheaf property.}
an $LG$-equivariant isomorphism
\[
   \F \isoarrow \F'.
\]
We shall always identify \F and $\F'$ in this way.

%
%

\subsection{Schubert cells and varieties}
In this subsection we discuss Schubert cells and varieties in the affine flag variety.  For $w\in\wt W$, the associated \emph{Schubert cell $C_w$} is the reduced
$k$-subscheme 
\[
   C_w := L^+B \cdot \dot{w} \subset \F,
\]
where $\dot w$ is any representative of $w$ in $G(K)$.  The associated \emph{Schubert variety $S_w$} is the reduced closure of $C_w$ in \F.  Since $L^+ B \subset LG^\circ$, every Schubert cell and variety is contained entirely in $\F^\circ$ or entirely in $\tau\F^\circ$.  By \eqref{st:wtW=double_cosets}, $\wt W$ is in bijective correspondence with the set of Schubert cells in \F.  We have $\dim C_w = \dim S_w = \ell(w)$.

\subsection{Closure relations between Schubert cells}\label{ss:closure_relns}
We now discuss closure relations between Schubert cells in \F.  In the case of a connected reductive group over $K$, closure relations between Schubert cells correspond exactly to the Bruhat order in the Iwahori-Weyl group.  Our aim here is to show that this statement carries over to our disconnected group $G$.

\begin{prop}\label{st:BO_closure_relns}
Let $w$, $w' \in \wt W$.  Then $w \leq w'$ in the Bruhat order $\iff$ $S_w
\subset S_{w'}$ in \F.
\end{prop}

\begin{proof}
We reduce to the analogous statement for $G^\circ$, using the decompositions $\wt W = \wt W^\circ \amalg \tau \wt W^\circ$ and $\F = \F^\circ \amalg \tau\F^\circ$. 
Let $w$, $w'\in\wt W$.  Then for $w$ and $w'$ to be related in the Bruhat order on the one hand, and for $S_w$ and $S_{w'}$ to be contained both in $\F^\circ$ or both in $\tau\F^\circ$ on the other hand, we must at least have $w \equiv w' \bmod \wt W^\circ$.  So we suppose this is the case.

If $w$, $w'\in \wt W^\circ$, then the conclusion follows at once from the lemma for $G^\circ$. If $w$, $w'\in \tau \wt W^\circ$, then we observe that
\begin{itemize}
\item
   the left-multiplication-by-$\tau$ map $\wt W^\circ \isoarrow \tau\wt W^\circ$ respects the Bruhat order, since $\tau$ stabilizes $A$; and
\item
   the left-multiplication-by-$\tau$ map $\F^\circ \isoarrow \tau\F^\circ$ respects Schubert cells, since $\tau$ normalizes $L^+B$ in $LG$.
\end{itemize}
So the conclusion in this case follows again from the statement for $G^\circ$.
\end{proof}

\section{Embedding the special fiber in the affine flag variety}\label{s:embedding}
In this section we embed the special fiber of $M^\naive$ into the affine flag variety \F.  

\subsection{The map}
We write $M^\naive_k := M^\naive \tensor_\O k$.  The embedding $M^\naive_k \inj \F$ we wish to construct will make use of the lattice-theoretic description of \F from \s\ref{ss:latt_desc}.  We first note that the $\O_K$-lattice chain $\lambda_\bullet$ admits a ``trivialization'' in obvious analogy with \eqref{disp:latt_triv}, where $\lambda_i$ replaces $\Lambda_i$, $\O_K$ replaces \O, and $t$ replaces $\pi$.  Then this trivialization together with  \eqref{disp:latt_triv} itself and the canonical identifications $\O/\pi\O \ciso k \ciso \O_K/t\O_K$ yields an identification of chains of $k$-vector spaces
\[\tag{$*$}\label{disp:chain_isom}
   \Lambda_\bullet \tensor_\O k \ciso \lambda_\bullet \tensor_{\O_K} k.
\]

To define $M^\naive_k \inj \F$, suppose we have an $R$-point $\{\F_i \inj \Lambda_i \tensor_\O R\}$ of $M^\naive_k$ for some $k$-algebra $R$.  Let $L_i \subset \lambda_i\tensor_{\O_K} R[[t]]$ be the submodule rendering the diagram
\[
   \xymatrix{
      L_i \ar@{^{(}->}[r] \ar@{->>}[d] & \lambda_i \tensor_{\O_K} R[[t]] \ar@{->>}[d]\\
      \F_i \ar@{^{(}->}[r] & \Lambda_i \tensor_\O R\ciso (\lambda_i \tensor_{\O_K} k) \tensor_k R
   }
\]
Cartesian, where the identification in the bottom right corner is made via \eqref{disp:chain_isom}.  Then the $L_i$'s form an indexed $R[[t]]$-lattice chain
\[
   L_0 \subset L_1 \subset \dotsb \subset L_{2n-1} \subset t^{-1}L_0
\]
in $R((t))^{2n}$.  The chain extends periodically to an $R$-point of \F (we may globally take the scalar $\alpha$ discussed in \s\ref{ss:latt_desc} to equal $t^{-1}$), which we take to be the image of our original $R$-point of $M^\naive_k$.  It is clear that $M^\naive_k \inj \F$ is then a monomorphism, and, as $M^\naive$ is proper, the map is a closed immersion.  From now on, we frequently identify $M_k^\naive$ with its image in \F.

\subsection{The image of the special fiber}
Let $R$ be a $k$-algebra.  It is clear from the definition of the map $M_k^\naive \inj \F$ that the image of $M^\naive_k(R)$ in $\F(R)$ consists precisely of all complete periodic self-dual chains $L_\bullet$ in $\F(R)$ such that, for all $i$,
\begin{itemize}
\item
   $\lambda_{i,R[[t]]} \supset L_i \supset t\lambda_{i,R[[t]]}$, where $\lambda_{i,R[[t]]} := \lambda_i \tensor_{\O_K} R[[t]]$; and
\item
   the $R$-module $\lambda_{i,R[[t]]}/L_i$ is locally free of rank $n$ for all $i$.
\end{itemize}

It is clear from this that the action of $L^+B$ on \F preserves the closed subschemes $M^\naive_k$ and $M^\spin_k$. We deduce that \emph{the underlying topological spaces of $M^\naive_k$ and $M^\spin_k$ are unions of Schubert varieties in \F.}  One of our essential goals for the rest of the paper is to obtain a good description of the Schubert varieties that occur in $M^\spin_k$.

\subsection{Schubert varieties in \texorpdfstring{$M_k^\naive$}{M\^{}naive}\_k}\label{ss:Sch_vties_Mnaive}
As a preliminary step towards describing the Schubert varieties $S_w$ that occur in $M_k^\spin$, in this subsection we translate the condition that $S_w$ be contained in the image of $M^\naive_k$ in \F into a condition on the extended alcove $v_0,\dotsc,v_{2n-1}$ attached to $w \in \wt W$ (\s\ref{ss:extd_alcs}).

Upon inspecting definitions, the previous subsection makes plain that $S_w$ is contained in $M^\naive_k$ $\iff$
\begin{enumerate}
\renewcommand{\theenumi}{P\arabic{enumi}}
\item\label{it:ineq_cond}
   $\omega_i \leq v_i \leq \omega_i + (1,\dotsc,1)$ for all $0\leq i \leq 2n-1$; and
\item\label{it:size_cond}
   $\Sigma v_0 = n$.
\end{enumerate}
We say that such a $w$ is \emph{$GL$-permissible}.  If $w$ is $GL$-permissible, then necessarily $d=0$ in the duality condition \eqref{it:dual_cond}, and it follows from the duality condition that the inequalities in \eqref{it:ineq_cond} hold for all $i$ as soon as they hold for all $0 \leq i \leq n$.  The condition that $w$ be $GL$-permissible is exactly the condition that it be permissible in $\wt W_{GL_{2n}}$ relative to the cocharacter $(1^{(n)},0^{(n)})$, or that, modulo conventions, its associated extended alcove be minuscule of size $n$ in the terminology of \cite{kottrap00}.

Given a $GL$-permissible $w$, the point $w\cdot \lambda_\bullet$ in $\F(k)$ corresponds to a point $\{\F_i \subset \Lambda_i \tensor_\O k\}$ in $M_k^\naive(k)$ of a rather special sort:  namely, identifying $\Lambda_i \tensor_\O k$ with $k^{2n}$ via \eqref{disp:latt_triv}, we have
\begin{enumerate}
\renewcommand{\theenumi}{T}
\item\label{disp:T_fixed}
   $\F_i$ is spanned by standard basis vectors in $k^{2n}$ for all $i$.
\end{enumerate}
On the other hand, for any point $\{\F_i\}$ in $M_k^\naive(k)$, let us say that $\{\F_i\}$ is a \emph{$T$-fixed point} if it satisfies \eqref{disp:T_fixed}; it is easy to check that the $T$-fixed points are exactly the points in $M_k^\naive(k)$ fixed by $L^+T(k)$.  In this way, we get a bijection between the $GL$-permissible $w\in \wt W$ and the $T$-fixed points in $M_k^\naive(k)$.

The $T$-fixed point $\{\F_i^w\}$ associated with a $GL$-permissible $w$ is easy to describe in terms of the extended alcove $v_0,\dotsc,v_{2n-1}$.  Indeed, let
\begin{equation}\label{disp:mu_i}
   \mu^w_i := v_i - \omega_i, \qquad 0 \leq i \leq 2n-1.
\end{equation}
Then $\mu_i^w$ is a vector in $\ZZ^{2n}$ having $n$ entries equal to $0$ and $n$ entries equal to $1$, and
\begin{equation}\label{disp:F_i}
   \F_i^w = \sum_{\mu^w_i(j) = 0} k \epsilon_j \subset k^{2n},
\end{equation}
where $\epsilon_1,\dotsc,\epsilon_{2n}$ is the standard ordered basis in $k^{2n}$.

\subsection{\texorpdfstring{$T$}{T}-fixed points in \texorpdfstring{$M_k^\spin$}{M\^{}spin\_k}}\label{ss:coord_pts}
By the previous subsection, every Schubert cell in \F contained in $M^\naive_k$ contains a unique $T$-fixed point in $M^\naive_k$.  So to understand which Schubert cells are contained in $M^\spin_k$, we need to understand which $T$-fixed points satisfy the spin condition. This is the object of this subsection.

We begin by fixing some notation.  We continue to write $e_1,\dotsc,e_{2n}$ for the standard basis in $V$ and $\epsilon_1,\dotsc,\epsilon_{2n}$ for the standard basis in $k^{2n}$, and we identify $\Lambda_i$ with $\O^{2n}$, and hence $\Lambda_i \tensor k$ with $k^{2n}$, via \eqref{disp:latt_triv}.  Quite generally, for any subset $E \subset \{1,\dotsc,2n\}$, we define
\[
   kE := \sum_{j\in E} k\epsilon_j \subset k^{2n}.
\]
When $E$ has cardinality $n$, consider the wedge product, in increasing index order, of the $n$ standard basis vectors in $\O^{2n}$ indexed by the elements of $E$; we denote by $e_E^i \in\bigwedge^n_F V$ the image of this element under the map $\bigwedge_\O^{n} \O^{2n} \isoarrow \bigwedge_\O^n \Lambda_i \subset \bigwedge^n_F V$.  When $i = 0$, we have $e^0_E = e_E$ \eqref{disp:e_E}.


Now let $\{\F_i \subset k^{2n}\}$ be a $T$-fixed point in $M_k^\naive(k)$.  For each $i$, let $E_i \subset \{1,\dotsc,2n\}$ be the subset of indices $j$ such that $\epsilon_j \in \F_i$, so that $\F_i = kE_i$ and $\F_i^\perp = kE_i^\perp$ \eqref{disp:E^*_E^perp}.
%
To understand the spin condition for the $\F_i$'s, we need to get a good handle on the elements $e^i_{E_i}$ and $e^i_{E_i^\perp}$.  More precisely, let
\[
   d_i := \#(E_i \cap \{1,\dotsc,i\}) \quad\text{and}\quad d_i^\perp := \#(E_i^\perp \cap \{1,\dotsc,i\}).
\]
Then, referring again to \eqref{disp:e_E},
\[
   e^i_{E_i} = \frac 1{\pi^{d_i}} e_{E_i}
   \quad\text{and}\quad
   e^i_{E_i^\perp} = \frac 1 {\pi^{d_i^\perp}} e_{E_i^\perp},
\]
and we need to understand the integer $d_i^\perp - d_i$.

To proceed, we'll consider pairs of the form $(i,2n - i)$ simultaneously, so that we may assume $0 \leq i \leq n$.  Let
\[
   A_i := \{1,\dotsc,i, i^*,\dotsc,2n\}
   \quad\text{and}\quad
   B_i := \{i+1,\dotsc,2n-i\},
\]
so that we get an orthogonal decomposition $k^{2n} = kA_i \oplus kB_i$.  Since $\F_n$ is totally isotropic, $E_i$ cannot contain any pair of the form $j$, $j^*$ with $j \leq i$.  Hence we may write $A_i$ as a disjoint union
\[
   A_i = R_i \amalg S_i,
\]
where
\[
   R_i := \{\, j \in A_i \mid \text{exactly one of $j$, $j^*$ is in $E_i$}\,\}
   \quad\text{and}\quad
   S_i := \{\, j \in A_i \mid j,\ j^* \notin E_i\,\}.
\]
Plainly, the sets $R_i$ and $S_i$ have even cardinalities, say equal to $2r_i$ and $2s_i$, respectively.  We have
\[
   E_i \cap A_i = E_i \cap R_i
   \quad\text{and}\quad
   E_i^\perp \cap A_i = (E_i^\perp \cap R_i) \amalg S_i = (E_i \cap R_i) \amalg S_i.
\]
Hence
\[
   \#(E_i\cap A_i) = r_i \quad\text{and}\quad \#(E_i^\perp\cap A_i) = r_i + 2s_i.
\]

We now need a couple of lemmas.

\begin{lem}\label{st:tot_istrp}
For $1 \leq i \leq n$, the image of $\F_i$ in $\F_{2n-i}$ under the structure maps is totally isotropic.
\end{lem}

\begin{proof}
Since $\F_i$ is spanned by standard basis vectors, the image in question is contained in $\F_i \cap \F_{2n-i} = \F_i \cap \F_i^\perp$.
\end{proof}

The $i = 0$ version of \eqref{st:tot_istrp} is simply the statement that $\F_0 = \F_0^\perp$.

\begin{lem}
$\#(E_i\cap A_i) \leq i$.
\end{lem}

\begin{proof}
The intersection $\F_i \cap kA_i$ is the precisely the image of $\F_i$ in $\F_{2n-i}$ under the structure maps, hence is totally isotropic by the previous lemma.  Since the form on $k^{2n}$ restricts to a nondegenerate form on $kA_i$, we conclude
\[
   \#(E_i\cap A_i) = \dim \F_i \cap kA_i \leq \tfrac 1 2 \dim kA_i = i. \qedhere
\]
\end{proof}

The lemma leaves us with two cases to consider.

\emph{Case 1: $\#(E_i\cap A_i) < i$.}  Then $S_i \neq \emptyset$.  Hence
\[
   d_i^\perp - d_i = s_i > 0.
\]
Hence by \eqref{disp:spin_basis}, we have
\[
   e^i_{E_i} \pm \pi^{d_i^\perp-d_i} \sgn(\sigma_{E_i}) e^i_{E_i^\perp} \in
      \Bigl(\sideset{}{_\O^n}{\bigwedge} \Lambda_i\Bigr)_\pm,
\]
and the image of this element under the map $\bigl(\bigwedge_\O^n \Lambda_i\bigr)_\pm \tensor k \to \bigwedge_k^n \Lambda_i \tensor k$ spans the line $\bigwedge_k^n\F_i$.  Moreover, it is easy to check that $d_{2n-i}^\perp - d_{2n-i} = s_i$ as well, so that we similarly conclude $\bigwedge_k^n\F_{2n-i} \subset \im\bigl[\bigl(\bigwedge_\O^n \Lambda_{2n-i}\bigr)_\pm \tensor k \to \bigwedge_k^n \Lambda_{2n-i} \tensor k\bigr]$.

\emph{Case 2: $\#(E_i\cap A_i) = i$.}  We claim $E_i = E_i^\perp$, that is, $\F_i$ is a (maximal) totally isotropic subspace of $k^{2n}$.  Indeed, in this case $\F_i \cap kA_i$ is maximal totally isotropic in $kA_i$, and it suffices to show that $\F_i \cap kB_i$ is totally isotropic of dimension $n - i$.  For this, consider the structure map $f\colon \F_{2n-i} \to \F_i$.  Then $\im f$ is plainly contained in $k B_i$ and is totally isotropic by the argument in \eqref{st:tot_istrp}.  So it suffices, in turn, to show that $\ker f = \F_{2n-i} \cap k A_i$ has dimension $i$.  But
%
\[
   \F_{2n-i} \cap kA_i \subset (\F_i \cap kA_i)^\perp \cap kA_i = \F_i \cap kA_i,
\]
where the equality in the display follows from our case assumption, and the reverse inclusion $\F_i \cap kA_i \subset \F_{2n-i} \cap kA_i$ is trivial.
The claim follows.  We deduce that $e_{E_i}^i$ and $e_{E_{2n-i}}^{2n-i} = e_{E_i}^{2n-i}$ are scalar multiples of each other; and as in \cite{paprap09}*{\s7.1.4}, both are contained in the one of the submodules $\bigl(\bigwedge_\O^n \Lambda_i\bigr)_\pm$.

We obtain the following.
\begin{prop}\label{st:spin_equiv_conds}
Let $\{\F_i \subset \Lambda_i \tensor k\}$ be a $T$-fixed point in $M^\naive(k)$.  The following are equivalent.
\begin{enumerate}
\renewcommand{\theenumi}{\roman{enumi}}
\item\label{it:spin_cond}
   $\{\F_i \subset \Lambda_i \tensor k\}$ satisfies the spin condition.
\item\label{it:OGr_comp}
   Upon identifying the $\Lambda_i\tensor k$'s with $k^{2n}$ via \eqref{disp:latt_triv}, all the $\F_i$'s for $0 \leq i \leq n$ which are totally isotropic in $k^{2n}$ specify points on the same connected component of the orthogonal Grassmannian $\OGr(n,2n)$.
\item\label{it:codim_cond}
   Under the above identifications, whenever $\F_i$ and $\F_{i'}$ for $0 \leq i, i' \leq n$ are totally isotropic in $k^{2n}$, $\F_i \cap \F_{i'}$ has even codimension in $\F_i$ and $\F_{i'}$.
\item\label{it:Wcirc_cond}
   The sets $E_i$ for $0 \leq i \leq n$ for which $E_i = E_i^\perp$ are all $W^\circ$-conjugate under the natural action of $W^\circ$ on $\{1,\dotsc,2n\}$.
\end{enumerate}
\end{prop}

\begin{proof}
\eqref{it:spin_cond} $\Longleftrightarrow$ \eqref{it:OGr_comp} has already been explained.  \eqref{it:OGr_comp} $\Longleftrightarrow$ \eqref{it:codim_cond} is explained in \cite{paprap09}*{\s7.1.4}.  \eqref{it:OGr_comp} $\Longleftrightarrow$ \eqref{it:Wcirc_cond} is clear from the facts that the orthogonal group acts transitively on $\OGr$, and that the element $\tau$ (\s\ref{ss:tau}) interchanges the two components.
\end{proof}

\subsection{Schubert varieties in \texorpdfstring{$M_k^\spin$}{M\^{}spin\_k}}\label{ss:Sch_vties_Mspin}
We now use the previous subsection to express the condition that the Schubert variety $S_w$ attached to $w \in \wt W$ is contained in $M_k^\spin$.  Continuing from \s\ref{ss:Sch_vties_Mnaive}, we shall express this condition in terms of the extended alcove $v_0,\dotsc,v_{2n-1}$ attached to $w$.

Let $w$ be $GL$-permissible.  Then the condition we wish to formulate can be essentially read off from \eqref{st:spin_equiv_conds}.  Recall the vector $\mu_i^w$ \eqref{disp:mu_i} and the subspace $\F^w_i \subset k^{2n}$ \eqref{disp:F_i}.  We say $\mu^w_i$ is \emph{totally isotropic} if $\mu_i(j) = 1 - \mu_i(j^*)$ for all $j$, or equivalently if $\F_i^w$ is totally isotropic in $k^{2n}$.  It is now immediate from our considerations of $T$-fixed points and from \eqref{st:spin_equiv_conds} that $S_w$ is contained in $M_k^\spin$ $\iff$ $w$ is $GL$-permissible and, in addition, satisfies
\begin{enumerate}
\renewcommand{\theenumi}{P3}
\item\label{it:mu_spin_cond}
   (spin condition) the vectors $\mu_i^w$ for $0 \leq i \leq n$ which are totally isotropic are all $W^\circ$-conjugate.
\end{enumerate}

The following trivial reformulation of \eqref{it:mu_spin_cond} is sometimes convenient.  Borrowing our notation from the previous subsection, let $E_i^w \subset \{1,\dotsc,2n\}$ be the subset
\begin{equation}\label{disp:E_i}
   E_i^w := \{\,j \mid \mu_i(j) = 0\,\}.
\end{equation}
We say $E_i^w$ is \emph{totally isotropic} if $E_i^w = (E_i^w)^\perp$, or equivalently if $\mu_i^w$ is totally isotropic.  Then for $GL$-permissible $w$, condition \eqref{it:mu_spin_cond} is equivalent to
\begin{enumerate}
\renewcommand{\theenumi}{P3$'$}
\item\label{it:E_spin_cond}
   (spin condition$'$) The sets $E_i^w$ for $0 \leq i \leq n$ which are totally isotropic are all $W^\circ$-conjugate.
\end{enumerate}

We say that $w\in \wt W$ is \emph{spin-permissible} if it satisfies \eqref{it:ineq_cond}, \eqref{it:size_cond}, and \eqref{it:mu_spin_cond}, or equivalently \eqref{it:ineq_cond}, \eqref{it:size_cond}, and \eqref{it:E_spin_cond}.  It follows from the duality condition \eqref{it:dual_cond} that for $GL$-permissible $w$, the vectors $\mu_0^w$ and $\mu_n^w$ are always totally isotropic; but in general, even for spin-permissible $w$, the possibilities can range from these two being the only totally isotropic vectors to all the all $\mu_i^w$'s being totally isotropic.

It is useful to formulate a slight refinement of the notion of spin-permissible.  There are exactly two orbits for the action of $W^\circ$ on the set of totally isotropic vectors with $n$ entries equal to $0$ and $n$ entries equal to $1$, namely
\[
   W^\circ \mu_1 \quad\text{and}\quad W^\circ \mu_2,
\]
where $\mu_1 = (1^{(n)},0^{(n)})$ and $\mu_2 = (1^{(n-1)},0,1,0^{(n-1)})$ are the cocharacters of \eqref{disp:mu's}.  For $j = 1$, $2$, we say that $w$ is \emph{$\mu_j$-spin-permissible} if $w$ is $GL$-permissible and $\mu_i^w \in W^\circ \mu_j$ whenever $\mu_i^w$ is totally isotropic.  We write $\Permsp(\mu_j)$ for the set of $\mu_j$-spin-permissible elements in $\wt W$.  Thus the set of spin-permissible elements in $\wt W$ is the disjoint union $\Permsp(\mu_1) \amalg \Permsp(\mu_2)$.  We shall see in \s\ref{ss:perm_spin-perm} that for $j = 1$, $2$, $\Permsp(\mu_j)$ is precisely the \emph{$\mu_j$-permissible set} defined by Kottwitz and Rapoport \cite{kottrap00}.

\subsection{Topological flatness of \texorpdfstring{$M^\spin$}{M\^{}spin}}
We now come to the main result of the paper.  We again recall the dominant minuscule cocharacters $\mu_1$ and $\mu_2$ for $G$ from \eqref{disp:mu's}, and for any cocharacter $\mu \in X_*(T)$, we recall the $\mu$-admissible sets $\Adm(\mu)$ and $\Adm^\circ(\mu)$ from \s\ref{ss:mu_adm}.  Let $\A(\mu)$ denote the reduced union of Schubert varieties $\bigcup_{w\in\Adm^\circ\mu}S_w$ in the affine flag variety.

\begin{thm}\label{st:main_result}
\hfill
\begin{enumerate}
\renewcommand{\theenumi}{\roman{enumi}}
\item\label{it:Adm(mu)=P}
   $\Adm^\circ(\mu_1) = \Permsp(\mu_1)$ and $\Adm^\circ(\mu_2) = \Permsp(\mu_2)$.  In particular, the set $\Adm(\mu_1) = \Adm(\mu_2)$ equals $\Permsp(\mu_1) \amalg \Permsp(\mu_2)$.
\item\label{it:special_fiber}
   The underlying topological space of the special fiber $M^\spin_k$ coincides with $\A(\mu_1) \amalg \A(\mu_2)$ in \F.  In particular, $M^\spin_k$ has two isomorphic connected components, and the irreducible components of these are in respective bijective correspondence with $W^\circ \mu_1$ and $W^\circ \mu_2$.
\item\label{it:top_flat}
   The underlying topological space of $M^\spin$ is the closure of the generic fiber in $M^\naive$.  In particular, $M^\spin$ is topologically flat.
\end{enumerate}
\end{thm}

\begin{proof}
Assertion \eqref{it:special_fiber} follows immediately from \eqref{it:Adm(mu)=P} and \s\ref{ss:Sch_vties_Mspin}.  To prove \eqref{it:special_fiber}, we must show that the irreducible components of $M^\spin_k$ are in the closure of the generic fiber $M^\spin_F = M^\naive_F$ in $M^\naive$.  This follows from \eqref{it:special_fiber} by a more-or-less standard lifting argument.  By a lemma of G\"ortz \cite{goertz05}*{Lemma 2}, it suffices to show that each irreducible component in $M^\spin_k$
\begin{enumerate}
\item\label{it:dim_cond}
   has dimension equal to the dimension of $M^\spin_F$; and
\item\label{it:lift_cond}
   contains a closed point which is contained in no other irreducible component and which lifts to the generic fiber.
\end{enumerate}
For \eqref{it:dim_cond}, for $\mu \in W^\circ \mu_1 \cup W^\circ \mu_2$, one readily computes from the formula of Iwahori-Matsumoto \cite{iwamat65}*{Proposition 1.23}
\[
   \dim S_{t_\mu} 
      = \ell(t_\mu) 
      = \sum_{\substack{\text{positive}\\
                        \text{roots $\alpha$}}} |\langle \mu, \alpha\rangle |
      = \binom n 2.
\]
On the other hand, it is well-known that $M^\spin_F \ciso \OGr(n,2n)_F$ has dimension $\binom n 2$.  For \eqref{it:lift_cond}, for each $\mu \in W^\circ \mu_1 \cup W^\circ \mu_2$, we just take the $T$-fixed point $\{\F_i^{t_\mu} \subset k^{2n}\} \in M^\spin(k)$ attached to $\mu$ itself.  Then the $\F_i^{t_\mu}$'s are all equal and spanned by the standard basis vectors $\epsilon_j$ for which $\mu(j) = 0$, and we have the obvious lift $\bigl\{\wt \F_i^{t_\mu} \subset \O^{2n}\bigr\} \in M^\spin(\O)$ where $\wt \F_i^{t_\mu}$ is the span of the corresponding standard basis vectors in $\O^{2n}$ for all $i$.

It remains to prove \eqref{it:Adm(mu)=P}.  To prove the containments $\Adm^\circ(\mu_j) \subset \Permsp(\mu_j)$ for $j = 1$, $2$, we first note that since $M^\spin_k$ is closed in \F and the Bruhat order reflects closure relations between Schubert varieties \eqref{st:BO_closure_relns}, $\Permsp(\mu_j)$ is closed in the Bruhat order. Hence it suffices to show that $\Permsp(\mu_j)$ contains the maximal elements of $\Adm^\circ(\mu_j)$, that is, the $W^\circ$-conjugates of $t_{\mu_j}$ in $\wt W$, which is obvious. 

We are left to prove the containments $\Adm^\circ(\mu_j) \supset \Permsp(\mu_j)$ for $j = 1$, $2$ in \eqref{it:Adm(mu)=P}.  This is the main object of \s\ref{s:adm_perm_sets}.
\end{proof}

\section{Admissible, permissible, and spin-permissible sets}\label{s:adm_perm_sets}

Let $\mu \in \{\mu_1,\mu_2\}$.  In this section we complete the proof of part \eqref{it:Adm(mu)=P} of \eqref{st:main_result} by showing that $\Permsp(\mu) \subset \Adm^\circ(\mu)$.  In essence, this amounts to working through the argument of Kottwitz-Rapoport \cite{kottrap00}*{\s5} in the case of the orthogonal similitude group.  In the last subsection \s\ref{ss:perm_spin-perm}, we show that the notion of $\mu$-spin-permissibility (\s\ref{ss:Sch_vties_Mspin}) agrees with the notion of $\mu$-permissibility from \cite{kottrap00}.

\subsection{Strategy}
Our strategy for proving $\Permsp(\mu) \subset \Adm^\circ(\mu)$ is, in the large, the same strategy Kottwitz and Rapoport used to prove the analogous assertion for $GL_r$.  Namely, let $w \in \Permsp(\mu)$.  Then the asserted containment holds $\iff$ $w$ is a translation element, or $w$ is not a translation element and we can find a reflection $s \in W_\aff$ such that $sw \in \Permsp(\mu)$ and $sw > w$ in the Bruhat order.  In the $GL$ case, when $w$ is not a translation element, Kottwitz and Rapoport found an explicit affine root $\alpha$ such that the associated reflection had the desired properties.  Since every affine root for $G$ is the restriction of an affine root for $GL_{2n}$, we can approach the problem in our case in the following way: regarding $w$ as an element in $\wt W_{GL_{2n}}$, we can take the affine root $\alpha$ prescribed by Kottwitz and Rapoport, attempt to restrict $\alpha$ to the maximal torus $T$ in $G$, and then take the corresponding reflection in $W_\aff$.  Two problems arise.
\begin{enumerate}
\item
   $\alpha$ may not restrict to an affine root of $G$.
\item
   Even when $\alpha$ does restrict to an affine root of $G$ with associated reflection $s_\alpha$, although one can show that $s_\alpha w > w$ and that $s_\alpha w$ satisfies \eqref{it:ineq_cond} and \eqref{it:size_cond}, \emph{$s_\alpha w$ need not satisfy \eqref{it:mu_spin_cond}.}
\end{enumerate}

It turns out that the first problem is quite easy to overcome.  But the second is more serious and leads us to a more complicated case analysis than that encountered in \cite{kottrap00}.

\subsection{Reflections}
Consider the affine linear function
\[
   \alpha_{i,j;d}\colon
   \xymatrix@R=0ex{
      X_*(T) \ar[r] & \ZZ\vphantom{X_*(T)}\\
      (x_1,\dotsc,x_{2n})\vphantom{x_i - x_j - d} \ar@{|->}[r] & x_i - x_j - d \vphantom{(x_1,\dotsc,x_{2n})}
   }
\]
for $i < j$ and $d\in\ZZ$.  Then $\alpha := \alpha_{i,j;d}$ is an affine root of $(G,T)$ precisely when $j \neq i^*$, and up to sign, all affine roots are obtained in this way.  Plainly $\alpha_{i,j;d} = \alpha_{j^*,i^*;d}$.  Attached to $\alpha$ is the reflection $s_\alpha = s_{i,j;d} \in W_\aff$ which acts on $X_*(T) \tensor \RR$ by sending $(x_1,\dotsc,x_n)$ to the tuple with $x_j + d$ in the $i$th slot, $x_i -d$ in the $j$th slot, $x_{i^*} + d$ in the $j^*$th slot, $x_{j^*} - d$ in the $i^*$th slot, and all other slots the same; visually, in the case $i < j < j^* < i^*$,
\begin{multline*}
   (\dotsc,x_i,\dotsc,x_j,\dotsc,x_{j^*},\dotsc,x_{i^*}\dotsc)
   \overset{s_\alpha}{\mapsto}\\
   (\dotsc,x_j+d,\dotsc,x_i-d,\dotsc,x_{i^*} + d,\dotsc,x_{j^*}-d,\dotsc).
\end{multline*}

If $w\in\wt W$ has extended alcove $v_0,\dotsc,v_{2n-1}$, then $s_\alpha w$ has extended alcove $s_\alpha v_0,\dotsc,s_\alpha v_{2n-1}$.

\subsection{The set \texorpdfstring{$K_m$}{K\_m}}\label{ss:K_m}
Fix a $GL$-permissible $w \in \wt W$, and recall the vector $\mu_k^w$ for $0 \leq k \leq 2n-1$ from \eqref{disp:mu_i}.  As in \cite{kottrap00}, for $1 \leq m \leq 2n$, we define $K_m \subset \{0,\dotsc,2n-1\}$ to be the subset
\[
   K_m := \{\, k \mid \mu_k^w(m) = 0\,\}.
\]

Just as in \cite{kottrap00}*{5.4--5.5}, and in the notation and terminology used there, the set $K_m$ is either empty, all of $\{0,\dotsc,2n-1\} \ciso \ZZ/2n\ZZ$, or an \emph{interval} in $\ZZ/2n\ZZ$ of the form $[\wt m,m)$ for some $\wt m \neq m$; in this last case, we say that $m$ is \emph{proper} and that $K_m$ has \emph{lower endpoint} $\wt m$ and \emph{upper endpoint} $m$.  For proper $m$, we always denote by $\wt m$ the lower endpoint of $K_m$.  Of course, the lower endpoint $\wt m \in \ZZ/2n\ZZ$ is characterized by the property
\[
   \wt m \in K_m \quad\text{and}\quad \wt m - 1 \notin K_m.
\]
When $m$ is proper, $\wt m$ is evidently proper too, and we have the simple formula
\begin{equation}\label{disp:E_fmla}
   E_{\wt m}^w = (m, \wt m) \cdot E_{\wt m - 1}^w,
\end{equation}
where $(m, \wt m)$ is the transposition interchanging $m$ and $\wt m$ and the set $E_{\wt m}^w$ is defined in \eqref{disp:E_i}.  Plainly, the function $m \mapsto \wt m$ defines a fixed-point-free bijection from the set of proper elements in $\ZZ/2n\ZZ$ to itself.  Note that, asymmetrically, we embed $w$ into our notation for $E^w_i$ but suppress $w$ in our notation for $K_m$.

The duality condition \eqref{it:dual_cond} may be expressed in terms of the $E^w$'s as
\[
   m \in E_k^w \iff m^* \notin E_{2n-k}^w
\]
and in terms of the $K$'s as
\[
   k \in K_m \iff 2n-k \notin K_{m^*}.
\]
Hence $K_{m^*} = -K_m^c$ for all $m$ as subsets of $\ZZ/2n\ZZ$.  Hence if $K_m$ is an interval $[\wt m,m)$, then $K_{m^*}$ is the interval $[(\wt m)^*,m^*)$; in particular, $(m^*)\sptilde = (\wt m)^*$.  Moreover, $m$ fails to be proper exactly when $K_m = \ZZ/2n\ZZ$ and $K_{m^*} = \emptyset$, or $K_m = \emptyset$ and $K_{m^*} = \ZZ/2n\ZZ$.

\subsection{Reflections and \texorpdfstring{$GL$}{GL}-permissibility}\label{ss:refl_GL-perm}
Suppose $w\in\wt W$ is $GL$-permissible.  In this subsection we determine the affine roots $\alpha$ such that $s_\alpha w$ is again $GL$-permissible.  As usual, we denote by $v_0,\dotsc,v_{2n-1}$ the extended alcove attached to $w$.

Recall from \s\ref{ss:Sch_vties_Mnaive} that $w$ is $GL$-permissible $\iff$ $\Sigma v_0 = n$ and $\omega_k \leq v_k \leq \omega_k + (1^{(2n)})$ for all $0 \leq k \leq 2n-1$.  Hence for $\alpha = \alpha_{i,j;d}$ with $i < j \neq i^*$, the element $s_\alpha w$ is $GL$-permissible $\iff$
\[\tag{$*$}\label{disp:GL2n_perm_conds}
   \left.
   \begin{gathered}
      v_k(j) + d - \omega_k(i) \in \{0,1\}\\
      v_k(i) - d - \omega_k(j) \in \{0,1\}\\
      v_k(i^*) + d - \omega_k(j^*) \in \{0,1\}\\
      v_k(j^*) - d - \omega_k(i^*) \in \{0,1\}
   \end{gathered}\:
   \right\}
   \quad\text{for all $0 \leq k \leq 2n-1$.}
\]
By the duality condition, the last two containments in \eqref{disp:GL2n_perm_conds} hold for all $k$ $\iff$ the first two hold for all $k$.

It is convenient to express the conditions in \eqref{disp:GL2n_perm_conds} in terms of the sets $[i,j)$, $K_i$, and $K_j$.  For any subset $S \subset \{0,\dotsc,2n-1\}$, let $\chi_S$ denote the characteristic function of $S$.  Then for all $k$ and $m$,
\[
   \chi_{[i,j)}(k) = \omega_k(j) - \omega_k(i) 
   \quad\text{and}\quad
   \chi_{K_m}(k) = 1 - \mu_k^w(m) = 1 -v_k(m) + \omega_k(m).
\]
Hence we may rewrite the first two conditions in \eqref{disp:GL2n_perm_conds} as
\[
   \chi_{K_j}(k) - \chi_{[i,j)}(k) - d \in \{0,1\}
   \qquad\text{and}\qquad
   \chi_{K_i}(k) + \chi_{[i,j)}(k) + d \in \{0,1\}
\]
for all $0 \leq k \leq 2n-1$.  Similarly to \cite{kottrap00}*{5.2}, either of these last two conditions implies that $d$ equals $0$ or $-1$.  We similarly conclude from the two conditions together that for $d = 0$,
\[
   \text{$s_{i,j;0}w$ is $GL$-permissible} \iff [i,j) \subset K_i^c \cap K_j,
\]
and that for $d = -1$,
\[
   \text{$s_{i,j;-1}w$ is $GL$-permissible} \iff [i,j)^c \subset K_i \cap K_j^c.
\]

The following is a convenient reformulation of the above discussion.

\begin{lem}\label{st:GL-perm_cond}
Let $i$, $j \in \{1,\dotsc,2n\}$ with $j \neq i$, $i^*$.  Then
\[
   \begin{varwidth}{\textwidth}
      \centering
      either $i < j$ and $s_{i,j;0} w$ is $GL$-permissible,\\
      or $j < i$ and $s_{j,i;-1}$ is $GL$-permissible
   \end{varwidth}
   \iff
   i \in K_j \text{ and } j-1 \notin K_i.
\]
\end{lem}

\begin{proof}
This is clear from the above discussion and the fact that $K_i$, resp.\ $K_j$, is either empty, all of $\ZZ/2n\ZZ$, or an interval with upper endpoint $i$, resp.\ $j$. 
\end{proof}

\subsection{Reflections and the Bruhat order}\label{ss:refl_BO}
We continue with our $w \in \wt W$ and affine root $\alpha = \alpha_{i,j;d}$ with $i < j \neq i^*$.  The elements $w$ and $s_\alpha w$ are related in the Bruhat order, and we have $w < s_\alpha w$ exactly when our base alcove $A$ and the alcove $wA$ lie on the same side of the hyperplane in \aaa where $\alpha$ vanishes.  We wish to understand this condition in terms of $\alpha$ and the extended alcove attached to $w$.

Actually, instead of working directly with $A$, it will be more convenient to use the analogous alcove $A'$ for the symplectic group: this is the interior of the convex hull in \aaa of the $n+1$ points
\[
   a_k' := \frac{\omega_k + \omega_{2n-k}}2 \mod\RR\cdot (1,\dotsc,1) 
   \quad\text{for}\quad
   0 \leq k \leq n.
\]
Then $A' \subset A$, so that it suffices to use $A'$ and $wA'$ to detect the Bruhat relation between $w$ and $s_\alpha w$.  The vertices of $wA'$ are
\[
   wa_k' = \frac{v_k + v_{2n-k}}2 \mod\RR \cdot (1,\dotsc,1)
   \quad\text{for}\quad
   0 \leq k \leq n.
\]
Hence
\[\tag{$*$}\label{disp:alpha_vals}
   \begin{aligned}
      \alpha(wa_k') 
         &= \frac{\chi_{K_j}(k) - \chi_{K_i}(k) - \chi_{[i,j)}(k)}2 \\
         &\qquad\qquad\qquad\quad + \frac{\chi_{K_j}(2n-k) - \chi_{K_i}(2n-k) - \chi_{[i,j)}(2n-k)}2 -d
   \end{aligned}
\]
for $0 \leq k \leq n$.

When $d \geq 0$, the values of $\alpha$ on the vertices of $A'$ are nonpositive.  Hence, in this case,
\[
   w < s_\alpha w
   \iff
   \text{the value in \eqref{disp:alpha_vals} is negative for some $k$.}
\]
On the other hand, when $d \leq -1$, the values of $\alpha$ on the vertices of $A'$ are nonnegative.  Hence, in this case,
\[
   w < s_\alpha w
   \iff
   \text{the value in \eqref{disp:alpha_vals} is positive for some $k$.}
\]

The following lemma builds on \eqref{st:GL-perm_cond} to give a useful characterization of when $s_\alpha w$ is $GL$-permissible and $w < s_\alpha w$.

\begin{lem}\label{st:GL-perm_BO_cond}
Let $i$, $j \in \{1,\dotsc,2n\}$ with $j \neq i$, $i^*$.  Suppose that $i$ is proper, so that $K_i$ is an interval $[\wt\imath,i)$ with $\wt\imath \neq i$.  Then
\[
   \begin{varwidth}{\textwidth}
      \centering
      either $i < j$, $s_{i,j;0}w$ is $GL$-permissible, and $w < s_{i,j;0} w$;\\
      or $j < i$, $s_{j,i;-1}w$ is $GL$-permissible, and $w < s_{j,i;-1} w$
   \end{varwidth}
   \iff
   i \in K_j \text{ and } \wt\imath \notin K_j.
\]
\end{lem}

\begin{proof}
We'll only need to use the implication ``$\Longleftarrow$'' later on, so we'll just prove that and leave the implication ``$\Longrightarrow$'' to the reader.  Let $\alpha$ denote the affine root $\alpha_{i,j;0}$ or $\alpha_{j,i;-1}$ according as $i < j$ or $j < i$.

We first address $GL$-permissibility.  By \eqref{st:GL-perm_cond}, regardless of the ordering of $i$ and $j$, we must show $j - 1 \notin K_i$.  But our hypotheses $i \in K_j$ and $\wt\imath \notin K_j$ clearly imply $j-1 \in [i,\wt\imath) = K_i^c$, where the superscript $c$ denotes the complement in $\ZZ/2n\ZZ$, as desired.

So it remains to show $w < s_\alpha w$.  We first suppose $i < j$, which leads us to look at the expression $\chi_{K_j} - \chi_{K_i} - \chi_{[i,j)}$.  Since $j$ is plainly proper by hypothesis, $K_j$ is an interval $[\wt\jmath,j)$ for some $\wt\jmath\neq j$.  Since $i \in K_j$, we have $[i,j) \subset K_j$, and
\[
   \chi_{K_{j}} - \chi_{[i,j)} = \chi_{[\wt\jmath,i)},
\]
where we interpret $[\wt\jmath,i) = \emptyset$ if $\wt\jmath = i$.  Moreover, since $\wt\imath \notin K_j$ and $i \in K_j$, we have $\wt\jmath \in [\wt\imath + 1 , i+1)$.  Hence $\wt\jmath - 1 \in K_i$.  Hence
\[
   \chi_{K_{j}} - \chi_{[i,j)} - \chi_{K_i} = -\chi_{[\wt\imath,\wt\jmath)}.
\]
Note that here $\wt\imath \neq \wt\jmath$ by injectivity of the map $m \mapsto \wt m$.  Hence
\[
   \alpha(wa_k') = \frac{-\chi_{[\wt\imath,\wt\jmath)}(k) - \chi_{[\wt\imath,\wt\jmath)}(2n-k)}{2}
\]
is certainly negative for some $0 \leq k \leq n$, as desired.

In the case $j < i$ with $\alpha = \alpha_{j,i;-1}$, one must find a vertex of $wA'$ on which $\alpha$ is positive.  This time one considers the expression
\[
   \chi_{K_i} - \chi_{K_j} - \chi_{[j,i)} + 1 = \chi_{K_i} - \chi_{K_j} + \chi_{[i,j)},
\]
which by the above reasoning equals $\chi_{[\wt\imath, \wt\jmath)}$, and the rest of the proof goes through similarly.
\end{proof}

As an important application, we obtain the following lemma.

\begin{lem}\label{st:[wta,a)_lem}
Let $r \in \{1,\dotsc,2n\}$, and suppose that $K_r$ is an interval $[\wt r, r)$ for some $\wt r \neq r$, $r^*$.
\begin{enumerate}
\renewcommand{\theenumi}{\roman{enumi}}
\item\label{it:case1}
   If $[r,\wt r) \subset K_{\wt r}$ and $r < \wt r$, let $\alpha := \alpha_{r,\wt r;0}$.
\item\label{it:case2}
   If $[r,\wt r) \subset K_{\wt r}$ and $\wt r < r$, let $\alpha := \alpha_{\wt r, r;-1}$.
\item\label{it:case3}
   If $K_{\wt r} \subset [r,\wt r)$ and $r < \wt r$, let $\alpha := \alpha_{r,\wt r;-1}$.
\item\label{it:case4}
   If $K_{\wt r} \subset [r,\wt r)$ and $\wt r < r$, let $\alpha := \alpha_{\wt r, r;0}$.
\end{enumerate}
Then in each case, $s_\alpha w$ is $GL$-permissible and $w < s_\alpha w$.
\end{lem}

Note that, since $[r,\wt r)$ and $K_{\wt r}$ are both intervals with upper endpoint $\wt r$, the hypotheses in at least one of \eqref{it:case1}--\eqref{it:case4} will always be satisfied.  So the force of the lemma is that, provided $r$ is proper and $\wt r \neq r^*$, we always get an affine reflection that preserves $GL$-permissibility and increases length.

\begin{proof}[Proof of \eqref{st:[wta,a)_lem}]
We use \eqref{st:GL-perm_BO_cond}.  To handle \eqref{it:case1} and \eqref{it:case2}, we must show $r \in K_{\wt r}$ and $\wt r \notin K_{\wt r}$, both of which are obvious.  To handle \eqref{it:case3} and \eqref{it:case4}, we must show $\wt r \in K_r$ and $\hspace{0.2ex}\wt{\hspace{-0.2ex}\wt r} \notin K_r$.  The first of these is obvious, and the second follows from
\[
   \wt{\hspace{-0.2ex}\wt r} \in K_{\wt r} \subset [r,\wt r) = K_r^c. \qedhere
\]
\end{proof}

\subsection{Reflections and the spin condition}\label{ss:refl_spin}
We continue with our $w$ and $\alpha = \alpha_{i,j;d}$ with $i < j \neq i^*$.  We now suppose that $w$ and $s_\alpha w$ are $GL$-permissible, and we wish to relate the spin condition on $s_\alpha w$ to the spin condition on $w$.  By \s\ref{ss:refl_GL-perm}, we must have $d = 0$ or $d = -1$.  Let $l_1 < l_2 < l_3 < l_4$ denote the elements of the set $\{i,i^*,j,j^*\}$ in increasing order, 
and consider the sets $E_k^w$ and $E_k^{s_\alpha w}$ \eqref{disp:E_i} for $0 \leq k \leq n$.  It is easy to verify that for $1 \leq k < l_1$ and for $l_2 \leq k \leq n$, the sets $E_k^w$ and $E_k^{s_\alpha w}$ are equal or conjugate by the permutation $(i,j)(i^*,j^*)$.  Hence, if $w$ is $\mu$-spin-permissible, then we at least know that the totally isotropic $E_k^{s_\alpha w}$ for $k \in \{0,\dotsc,n\} \smallsetminus [l_1,l_2)$ are $W^\circ$-conjugate to $E^w_0$, and hence to $E^{t_\mu}_0$.

It is a more subtle matter to handle the $E_k^{s_\alpha w}$'s for $k \in [l_1,l_2)$.  Since $i < j \neq i^*$, there are four possibilities to consider:
\[
   i < j < j^* < i^*, \quad i < j^* < j < i^*, \quad j^* < i < i^* < j,\quad\text{or}\quad j^* < i^* < i < j.
\]
In each case, one element $\ext(i,j)$ of the pair $i$, $j$ is extremal amongst the four elements, and the other element $\int(i,j)$ of the pair is not; and ditto for the pair $i^*$, $j^*$.  For fixed $k \in [l_1,l_2)$, one verifies that either
\[
   \ext(i,j) \notin E_k^w,E_k^{s_\alpha w},
   \quad
   \int(i,j) \in E_k^w,E_k^{s_\alpha w},
   \quad\text{and}\quad
   E_k^{s_\alpha w} = (i^*,j^*) \cdot E_k^w;
\]
or
\[
   \ext(i,j)^* \notin E_k^w,E_k^{s_\alpha w},
   \quad
   \int(i,j)^* \in E_k^w,E_k^{s_\alpha w},
   \quad\text{and}\quad
   E_k^{s_\alpha w} = (i,j) \cdot E_k^w.
\]

\begin{eg}\label{eg:E's}
The following illustration of our discussion will come up explicitly in \s\ref{ss:pf}.  Assume that $i < j < j^* < i^*$ and that
\[
   i,j^* \in E_{i-1}^w, \quad
   i^*,j \notin E_{i-1}^w, \quad
   i,i^* \notin E_i^w, \quad\text{and}\quad
   j,j^* \in E_i^w.
\]
Then, displaying the $i$th, $j$th, $j^*$th, and $i^*$th entries,
\[
   v_{i-1} = (\dotsc,0,\dotsc,1,\dotsc,0,\dotsc,1,\dotsc)
   \quad\text{and}\quad
   v_i = (\dotsc,0,\dotsc,0,\dotsc,0,\dotsc,1,\dotsc).
\]
Hence
\begin{align*}
   s_{i,j;0}v_{i-1} &= (\dotsc,1,\dotsc,0,\dotsc,1,\dotsc,0,\dotsc),\\
   s_{i,j;0}v_i &= (\dotsc,0,\dotsc,0,\dotsc,1,\dotsc,0,\dotsc),\\
   s_{i,j;-1}v_{i-1} &= (\dotsc,0,\dotsc,1,\dotsc,0,\dotsc,1,\dotsc),
                                 \quad\text{and}\\
   s_{i,j;-1}v_i &= (\dotsc,-1,\dotsc,1,\dotsc,0,\dotsc,1,\dotsc).
\end{align*}
Hence for $\alpha = \alpha_{i,j;0}$, provided $s_\alpha w$ is $GL$-permissible, we conclude
\[
   E_{i-1}^{s_\alpha w} = (i,j)(i^*,j^*) E_{i-1}^w
   \quad\text{and}\quad
   E_i^{s_\alpha w} = (i^*,j^*) E_i^w;
\]
and for $\alpha = \alpha_{i,j;-1}$, provided $s_\alpha w$ is $GL$-permissible, we conclude
\[
   E_{i-1}^{s_\alpha w} = E_{i-1}^w
   \quad\text{and}\quad
   E_i^{s_\alpha w} = (i,j) E_i^w.
\]
Either way, we conclude $E_{i-1}^{s_\alpha w} = E_i^{s_\alpha w}$.  The same conclusions plainly hold if $i < j^* < j < i^*$.
\end{eg}

Part \eqref{st:spin_fail_i} of the following lemma summarizes the first paragraph of this subsection, and part \eqref{st:spin_fail_ii} is an immediate consequence of the second paragraph.

\begin{lem}\label{st:spin_fail}
Suppose that $w$ is $\mu$-spin-permissible and that $s_\alpha w$ is $GL$-permissi\-ble.
\begin{enumerate}
\renewcommand{\theenumi}{\roman{enumi}}
\item\label{st:spin_fail_i}
   $s_\alpha w$ fails to be $\mu$-spin-permissible $\iff$ there exists $k \in [l_1,l_2)$ such that $E_k^{s_\alpha w}$ is totally isotropic and not $W^\circ$-conjugate to $E_0^{s_\alpha w}$.
\item\label{st:spin_fail_ii}
   For $k\in[l_1,l_2)$, $E_k^{s_\alpha w}$ is totally isotropic $\iff$ $\ext(i,j)$, $\ext(i,j)^* \notin E_k^w$; $\int(i,j)$, $\int(i,j)^* \in E_k^w$; and for every $r\in\{1,\dotsc,2n\} \smallsetminus \{i,j,i^*,j^*\}$, the set $E_k^w$ contains exactly one element from the pair $r$, $r^*$.\qed
\end{enumerate}
\end{lem}

\subsection{Completion of the proof of ({\protect \ref{st:main_result})}}\label{ss:pf}
We now commence the proof proper that $\Permsp(\mu) \subset \Adm^\circ(\mu)$.  We assume from now on that $w$ is $\mu$-spin-permissible and not a translation element in $\wt W$, and we must find an affine root $\alpha$ such that $s_\alpha w$ is $\mu$-spin-permissible and $w < s_\alpha w$.

To say that $w$ is not a translation element is precisely to say that some element in $\{1,\dotsc,2n\}$ is \emph{proper}; let us denote by $a$ the minimal proper element in $\{1,\dotsc,2n\}$.  Then
\[
   E_0^w = E_1^w = \dotsb = E_{a-1}^w \neq E_a^w.
\]
Since $a$ is proper $\iff$ $a^*$ is proper, we have $a \leq n$, and $a^*$ is the maximal proper element in $\{1,\dotsc,2n\}$.  As usual, we have $K_a = [\wt a ,a)$ for some $\wt a \neq a$; and our minimality assumption implies $a < \wt a$.  We claim $\wt a \neq a^*$.  For suppose to the contrary that $\wt a = a^*$.  Since $a^*$ is the maximal proper element, we have $E_{a^*}^w = E_0^w$.  Hence $E_{a^*}$ is totally isotropic.  But $E_{a^*}^w = (a,a^*)\cdot E_{a^*-1}^w$ \eqref{disp:E_fmla}. Hence $E_{a^*-1}^w$ is totally isotropic too but not $W^\circ$-conjugate to $E_{a^*}^w$, in violation of the spin condition.

%

Since $\wt a \neq a^*$, \eqref{st:[wta,a)_lem}, applied with $r = a$, immediately furnishes an affine root $\alpha$ such that $s_\alpha w$ at least is $GL$-permissible and $w < s_\alpha w$.  Unfortunately, in general, $s_\alpha w$ need not satisfy the spin condition.  To modify our choice of $\alpha$ if necessary, we shall need to set up a case analysis.

Since $a$ is proper, the set $E_a^w \smallsetminus E_{a-1}^w$ consists of a single proper element $b$, and $K_b = [a,b)$.  Of course $b \neq a$; and it follows from the inequality $\wt a \neq a^*$ that $b \neq a^*$.  By minimality of $a$, we thus have $a < m < m^* < a^*$, where $m := \min\{b,b^*\}$.  Since $E_{a-1}^w$ is totally isotropic, it must contain $b^*$, and we conclude $b$, $b^* \in E_k^w$ for all $k \in [a,m)$.  Note that by taking $r = b$ in \eqref{st:[wta,a)_lem}, we again get an affine root $\alpha$ such that $s_\alpha w$ is $GL$-permissible and $w < s_\alpha w$, but we again have the problem that $s_\alpha w$ may not satisfy the spin condition.

We shall base our case analysis on the existence of proper elements in $[a,m)$ that satisfy certain conditions.  Note that if $i \in [a,m)$ is proper, then $\wt \imath$ and $\wt \imath^*$ are proper too.  Hence $a \leq \wt \imath, \wt \imath^* \leq a^*$.

\subsubsection{Case: There exists a proper $r \in [a,m)$ such that $\wt r \neq r^*$ and $\min\{\wt r, \wt r^*\} < m$.}\label{sss:caseA}
Then \eqref{st:[wta,a)_lem}, applied to the element $r$, furnishes an affine root $\alpha$ such that $s_\alpha w$ is $GL$-permissible and $w < s_\alpha w$.  To see that $s_\alpha w$ satisfies the spin condition, write $l_1 < l_2 < l_3 < l_4$ for the elements $r$, $r^*$, $\wt r$, $\wt r^*$ in increasing order.  Our case assumption implies $[l_1,l_2) \subset [a,m)$.  Hence $b$, $b^* \in E_k^w$ for all $k \in [l_1,l_2)$.  Hence, by \eqref{st:spin_fail}, $E_k^{s_\alpha w}$ is not totally isotropic for such $k$ and $s_\alpha w$ satisfies the spin condition.\bigskip

In the remaining two cases we shall assume there exists no proper $r \in [a,m)$ as in \eqref{sss:caseA}.  Hence for every proper $i \in [a,m)$ with $\wt\imath \neq i^*$, we have $m \leq \wt i, \wt i^* \leq m^*$.  In particular, we have $m \leq \wt a, \wt a^* \leq m^*$, so that $a$, $a^* \notin E_k^w$ for all $k\in [a,m)$.

\subsubsection{Case: There exists no $r$ as in \eqref{sss:caseA}, and there exists a proper $l \in [a,m)$ distinct from $a$}\label{sss:caseB}
In this case we have
\[
   a < l < m \leq \wt a, \wt a^* \leq m^* < l^* < a^*.
\]
We consider the possibilities $\wt l = l^*$ and $\wt l \neq l^*$ separately.

If $\wt l = l^*$, then $K_l = [l^*,l)$ and $K_{l^*} = [l,l^*)$.  Plainly $a$,~$a^* \in K_l$ and $\wt a$,~$\wt a^* \notin K_l$.  Hence $s_{a,l;0} w$ and $s_{l,a^*;-1} w$ are $GL$-permissible and $w < s_{a,l;0} w$, $s_{l,a^*;-1} w$ \eqref{st:GL-perm_BO_cond}.  Moreover, since $K_l$ and $K_{l^*}$ are disjoint, it is immediate from \eqref{st:spin_fail} that $s_{a,l;0} w$ and $s_{l,a^*;-1} w$ both satisfy the spin condition.  We remark that similar reasoning reveals that one can also use either of the reflections $s_{b,l^*;0}$ or $s_{b^*,l^*;0}$.

If $\wt l \neq l^*$, then $m \leq \wt l, \wt l^* \leq m^*$ by our case assumption, and we take $\alpha := \alpha_{l,m;0}$.  Plainly
\[
   l \in [a,m) \subset K_m  \quad\text{and}\quad \wt l^* -1 \in [a,m^*) \subset K_{m^*}.
\]
By duality, the second displayed containment implies $\wt l \notin K_m$.  Hence $s_\alpha w$ is $GL$-permissible and $w < s_\alpha w$ \eqref{st:GL-perm_BO_cond}.  To check the spin condition, recall that $a$, $a^* \notin E_k^w$ for all $k \in [a,m) \supset [l,m)$.  Hence, by \eqref{st:spin_fail}, $E_k^{s_\alpha w}$ is not totally isotropic for such $k$, and $s_\alpha w$ is $\mu$-spin-permissible.\bigskip

Having dispensed with the above two cases, we are left with just the following case to consider.

\subsubsection{Case: $a$ is the only proper element in $[a,m)$}
By taking $r = b$ in \eqref{st:[wta,a)_lem}, we have that $s_\alpha w$ is $GL$-permissible and $w < s_\alpha w$ for $\alpha := \alpha_{a,b;0}$ or $\alpha := \alpha_{a,b;-1}$.  Thus we reduce to proving the claim:
\[
   \text{\emph{If $s_\alpha w$ is $GL$-permissible for $\alpha \in \{\alpha_{a,b;0}, \alpha_{a,b;-1}\}$, then $s_\alpha w$ is spin-permissible.}}
\]
So suppose we have such an $\alpha$. Our minimality assumption on $a$ and our case assumption together imply
\[
   E_0^w = E_1^w = \dotsb = E_{a-1}^w \neq E_a^w = E_{a+1}^w = \dotsb = E_{m-1}^w.
\]
But this places us exactly in the situation of \eqref{eg:E's}, with $i = a$ and $j = b$.  Hence for either possible $\alpha$, we have equalities
\[
   E_0^{s_\alpha w} = E_1^{s_\alpha w} = \dotsb = E_{m-1}^{s_\alpha w}.
\]
Hence $s_\alpha w$ is $\mu$-spin-permissible by \eqref{st:spin_fail}.

This completes our case analysis, and with it the proof of \eqref{st:main_result}.\qed

\begin{rk}
Implicit in our proof is a slight simplification of part of the proof \cite{kottrap00}*{5.8} of the main result for $GL_n$ in Kottwitz's and Rapoport's paper.  Indeed, our Lemma \ref{st:[wta,a)_lem}, formulated without the requirement that $\wt r \neq r^*$, continues to hold in the $GL_n$ setting.  So, using the language of \cite{kottrap00}, if $w \in \wt W_{GL_n}$ has minuscule associated alcove $\mathbf v$ and is not a translation element, then there must exist a proper $r \in \{1,\dotsc,n\}$, and the lemma immediately furnishes an $\alpha$ such that $s_\alpha \mathbf v$ is minuscule and $w < s_\alpha w$.  On the other hand, \cite{kottrap00} actually proves a little more: namely, that $\alpha$ can always be chosen to satisfy the additional constraint that the translation parts of $w$ and $s_\alpha w$ are the same.  We can find such an $\alpha$ by letting $a$ denote the \emph{minimal} proper element in $\{1,\dotsc,n\}$; then $\alpha_{a,\wt a;0}$ or $\alpha_{a,\wt a;-1}$ does the job.

As noted by Kottwitz and Rapoport, it follows that
\[\tag{$*$}\label{disp:statement}
   \begin{varwidth}{\textwidth}
      \centering
      $w \in \wt W_{GL_n}$ is $\mu$-admissible\\
      for minuscule $\mu$
   \end{varwidth}
   \:\implies\:
   \begin{varwidth}{\textwidth}
      \centering
      $w$ is less than or equal to its\\
      translation part in the Bruhat order.
   \end{varwidth}
\]
Much more generally, Haines \cite{haines01b}*{proof of 4.6}, using Hecke algebra techniques, has shown that \eqref{disp:statement} continues to hold when $\wt W_{GL_n}$ is replaced by the extended affine Weyl group attached to any root datum.  Unfortunately, the arguments in this paper do not seem to yield a direct proof of \eqref{disp:statement} for $\wt W$ and $\mu \in \{\mu_1,\mu_2\}$.
\end{rk}

\subsection{Permissibility and spin-permissibility}\label{ss:perm_spin-perm}
We conclude the paper by showing that Kottwitz's and Rapoport's notion of $\mu$-permissibility \cite{kottrap00} agrees with our notion of $\mu$-spin-permissibility for elements in $\wt W$.  While we have only defined $\mu$-spin-permissibility for $\mu \in \{\mu_1,\mu_2\}$, the notion of $\mu$-permissibility makes sense for any cocharacter $\mu$:  quite generally, $w \in \wt W$ is \emph{$\mu$-permissible} if $w \equiv t_\mu \bmod W_\aff$ and $wx - x \in \Conv(W^\circ \mu)$ for all $x$ in $\wt A$, where $\Conv(W^\circ \mu)$ is the convex hull in $X_*(T) \tensor \RR$ of the $W^\circ$-conjugates of $\mu$, and $\wt A$ is the alcove in $X_*(T) \tensor \RR$ obtained as the inverse image of $A$.  Of course, it is equivalent to require $wx -x \in \Conv(W^\circ \mu)$ for all $x$ in the closure of $\wt A$.  We denote by $\Perm(\mu)$ the set of $\mu$-permissible elements.

\begin{prop}\label{st:perm_spin-perm}
$\Permsp(\mu) = \Perm(\mu)$ for $\mu \in \{\mu_1,\mu_2\}$.
\end{prop}

\begin{proof}
The containment $\subset$ follows from the equality $\Permsp(\mu) = \Adm^\circ(\mu)$ \eqref{st:main_result} and the general result \cite{kottrap00}*{11.2} that $\mu$-admissibility implies $\mu$-permissibility for any cocharacter $\mu$ in any extended affine Weyl group attached to a root datum.  (Note that while $\wt W$ is not the extended affine Weyl group attached to a root datum, $\wt W^\circ$ is, and the sets in question are all contained in $\wt W^\circ$.)  To prove the reverse containment, suppose $w \in \Perm(\mu)$.  Since $\Conv(W^\circ \mu)$ is contained in $\Conv(W \mu)$ (this is the relevant convex hull that comes up for $GSp_{2n}$), \cite{kottrap00}*{12.4} shows, modulo conventions, that \eqref{it:ineq_cond} and \eqref{it:size_cond} hold for $w$.  It remains to show that if the vector $\mu_k^w$ \eqref{disp:mu_i} is totally isotropic for $0 \leq k \leq n$, then $\mu_k^w \in W^\circ \mu$.  For any $k$, since $\frac{\omega_k + \omega_{2n-k}} 2$ is in the closure of $\wt A$, we have
\[
   \frac{\mu_k^w + \mu_{2n-k}^w} 2 \in \Conv(W^\circ \mu).
\]
But if $\mu_k^w$ is totally isotropic, then $\mu_k^w = \mu_{2n-k}^w$ and the displayed vector equals $\mu_k^w$.  Now use the obvious fact that $X_*(T) \cap \Conv(W^\circ \mu) = W^\circ \mu$.
\end{proof}

Although we didn't need it for the proof, it is not hard to give an explicit description of the convex hull $\Conv(W^\circ\mu)$.  We set $V := X_*(T) \tensor \RR$, and we identify it with
\[
   \{\, (x_1,\dotsc,x_{2n}) \in \RR^{2n} \mid x_1 + x_{2n} = x_2 + x_{2n-1} = \dotsb = x_n + x_{n+1} \,\}.
\]
For $x = (x_1,\dotsc,x_{2n}) \in V$, we write $c(x)$ for the common value $x_1 + x_{2n} = \dotsb = x_n + x_{n+1}$.  We write $x \cdot y$ for the usual dot product of vectors in $\RR^{2n}$.  Then it is readily verified that, when $n$ is odd,
\[
   \Conv(W^\circ \mu) = \biggl\{\, x \in V\, \biggm| 
      \begin{varwidth}{\textwidth}
         \centering
         $(0,\dotsc,0) \leq x \leq (1,\dotsc,1)$, $c(x) = 1$,\\
         and $\mu' \cdot x \geq 1$ for all $\mu' \in W^\circ \mu$
      \end{varwidth}
      \,\biggr\};
\]
and when $n$ is even,
\[
   \Conv(W^\circ \mu) = \biggl\{\, x \in V\, \biggm| 
      \begin{varwidth}{\textwidth}
         \centering
         $(0,\dotsc,0) \leq x \leq (1,\dotsc,1)$, $c(x) = 1$,\\
         and $\mu' \cdot x \geq 1$ for all $\mu' \in \tau W^\circ \mu$
      \end{varwidth}
      \,\biggr\},
\]
where $\tau$ is the usual element from \s\ref{ss:tau}.

\begin{rk}
For $i = 1$, $2$, let $Y_i$ denote the common set
\[
   \Adm^\circ(\mu_i) = \Permsp(\mu_i) = \Perm(\mu_i).
\]
Using a subscript $GL_{2n}$ to denote the corresponding notions for elements in $\wt W_{GL_{2n}}$, let $Z$ denote the common set
\[
   \Adm_{GL_{2n}}(\mu_1) = \Adm_{GL_{2n}}(\mu_2) = \Perm_{GL_{2n}}(\mu_1) = \Perm_{GL_{2n}}(\mu_2);
\]
here we use the equivalence between admissibility and permissibility for minuscule cocharacters in $GL_{2n}$ due to Kottwitz and Rapoport \cite{kottrap00}*{3.5}.  Then we have relations between $Y_1$, $Y_2$, and $Z$,
\[
   Y_1 \sqcup Y_2 \ctndneq Z \cap \wt W^\circ \ctndneq Z \cap \wt W.
\]
(Recall that we always assume $n \geq 2$; here the first $\ctndneq$ becomes an equality when $n = 1$.)
\end{rk}

\end{document}